\documentclass[10pt]{article}

\usepackage{amssymb,amsmath,amsthm,amsfonts}
\usepackage{graphicx,graphics}
\usepackage{epstopdf}
\usepackage{algorithm,algorithmic}
\graphicspath{{./fig/}}
\usepackage{verbatim}
\usepackage{indentfirst}
\usepackage{mathtools}
\usepackage[usenames,dvipsnames,svgnames,table]{xcolor}
\usepackage{threeparttable}
\usepackage{booktabs}

\usepackage{subfigure}
\usepackage{subfig}
\allowdisplaybreaks

\usepackage{array}
\usepackage[nohead,margin=1.1in]{geometry}
\usepackage{color}
\numberwithin{equation}{section}
\numberwithin{algorithm}{section}

\newtheorem{theorem}{Theorem}[section]
\newtheorem{lemma}{Lemma}[section]
\newtheorem{proposition}{Proposition}[section]
\newtheorem{corollary}{Corollary}[section]

\newtheorem{remark}{Remark}[section]

\allowdisplaybreaks

\def\BC{\mathbb C}
\def\BK{\mathbb K}
\def\BN{\mathbb N}
\def\BR{\mathbb R}
\def\cA{\mathcal A}
\def\d{\mathrm d}
\def\e{\mathrm e}
\def\Ga{\Gamma}
\def\Te{\Theta}
\def\Si{\Sigma}
\def\Om{\Omega}
\def\al{\alpha}
\def\be{\beta}
\def\ga{\gamma}
\def\de{\delta}
\def\ve{\varepsilon}
\def\te{\theta}
\def\ze{\zeta}
\def\la{\lambda}
\def\si{\sigma}
\def\vp{\varphi}
\def\om{\omega}
\def\f{\frac}
\def\nb{\nabla}
\def\ov{\overline}
\def\pa{\partial}
\def\wh{\widehat}
\def\wt{\widetilde}

\title{Recovery of Multiple Parameters in Subdiffusion from One Lateral Boundary Measurement\thanks{The work of B. Jin is supported by UK EPSRC grant EP/T000864/1 and EP/V026259/1, and a start-up fund from The Chinese University of Hong Kong. {The work of Y. Liu is supported by Grant-in-Aid for Early Career Scientists 20K14355 and 22K13954, JSPS.} The work of Z. Zhou is partly supported by Hong Kong Research Grants Council (15303122) and an internal grant of Hong Kong Polytechnic University (Project ID: P0038888, Work Programme: ZVX3)}}
\author{Siyu Cen\thanks{Department of Applied Mathematics, The Hong Kong Polytechnic University, Kowloon, Hong Kong, P.R. China. (\texttt{siyu2021.cen@connect.polyu.hk; zhizhou@polyu.edu.hk})} \and Bangti Jin\thanks{Department of Mathematics, The Chinese University of Hong Kong, Shatin, New Territories, Hong Kong, P.R. China (\texttt{bangti.jin@gmail.com, b.jin@cuhk.edu.hk}).} \and Yikan Liu\thanks{Research Center of Mathematics for Social Creativity, Research Institute for Electronic Science, Hokkaido University, N12W7, Kita-Ward, Sapporo 060-0812, Japan (\texttt{ykliu@es.hokudai.ac.jp})} \and Zhi Zhou\footnotemark[2]}
\date{\today}

\begin{document}

\maketitle

\begin{abstract}
This work is concerned with numerically recovering multiple parameters simultaneously in the subdiffusion model
from one single lateral measurement on a part of the boundary, while in an incompletely known medium. We prove
that the boundary
measurement corresponding to a fairly general boundary excitation uniquely determines the order of the fractional
derivative and the polygonal support of the diffusion coefficient, without knowing either the initial condition or the source.
The uniqueness analysis further inspires the development of a robust numerical algorithm for recovering the
fractional order and diffusion coefficient. The proposed algorithm combines small-time asymptotic
expansion, analytic continuation of the solution and the level set method. We present extensive numerical experiments
to illustrate the feasibility of the simultaneous recovery. In addition, we discuss the uniqueness of recovering
general diffusion and potential coefficients from one single partial boundary measurement, when the boundary
excitation is more specialized.

\end{abstract}



\section{Introduction}
This work is concerned with an inverse problem of simultaneously recovering multiple parameters
in a subdiffusion model from one single lateral boundary measurement in a partly unknown medium.
Let $\Om\subset\BR^d$ ($d=2,3$) be an open bounded domain with a Lipschitz and
piecewise $C^{1,1}$ boundary and $T>0$ be a fixed final time. Consider the following
subdiffusion problem for the function $u$:
\begin{align}\label{eqn:fde}
\begin{cases}
\pa_t^\al u+\cA u=f & \mbox{in }\Om\times(0,T],\\
a\pa_\nu u=g & \mbox{on }\pa\Om\times(0,T],\\
u(0)=u_0 & \mbox{in }\Om,
\end{cases}
\end{align}
where $u_0\in L^2(\Om)$ and (time-independent) $f\in L^2(\Om)$ are unknown initial and
source data, and $\nu$ denotes the unit outward normal vector to the boundary $\pa\Om$. The
elliptic operator $\cA$ is defined by
\begin{align*}
\cA u(x):=-\nb\cdot(a(x)\nb u(x)),\quad x\in\ov\Om.
\end{align*}
Without loss of generality, the diffusion coefficient $a$ is assumed to be piecewise constant:
\begin{align}\label{eqn:a}
a(x)=1+\mu\,\chi_D(x),
\end{align}
where {$\mu>-1$ is a nonzero unknown constant, $D$ is an unknown convex polyhedron in $\Om$
satisfying $\mathrm{diam}(D)<\mathrm{dist}(D,\pa\Om)$ and $\chi_D$ denotes the characteristic function of $D$}. In the model \eqref{eqn:fde},
$\pa_t^\al u$ denotes the Djrbashian-Caputo fractional derivative in time $t$ of order
$\al\in(0,1)$ defined by (\cite[p. 92]{KilbasSrivastavaTrujillo:2006} or \cite[Section 2.3]{Jin:book2021})
\begin{align*}
\pa_t^\al u{(t):}=\f1{\Ga(1-\al)}\int_0^t(t-s)^{-\al}u'(s)\,\d s.
\end{align*}

The model \eqref{eqn:fde} has attracted a lot of recent attention, due to its excellent
capability to describe anomalous diffusion phenomena observed in many engineering and physical
applications. The list of successful applications is long and still fast growing, e.g.,
ion transport in column experiments \cite{HatanoHatano:1998},
protein diffusion within cells \cite{Golding:2006} and contaminant transport in underground
water \cite{Kirchner:2000}. See the reviews \cite{MetzlerKlafter:2000,MetzlerJeon:2014}
for the derivation of relevant mathematical models and diverse applications. The model
\eqref{eqn:fde} differs considerably from the normal diffusion model due to the presence
of the nonlocal operator $\pa_t^\al u$: it has limited smoothing property
in space and slow asymptotic decay at large time \cite{KubicaYamamoto:2020,Jin:book2021}.

In this paper, we study mathematical and numerical aspects of an inverse problem of
recovering the diffusion coefficient $a$ and fractional order $\al$ from a single lateral boundary measurement
of the solution, without the knowledge of the initial data $u_0$ and source $f$. The Neumann data $g$ is taken to be separable:
\begin{align}\label{eq:g}
g(x,t)=\psi(t)\eta(x),
\end{align}
where $0\not\equiv\eta\in H^{\frac{1}{2}}(\pa\Om)$ satisfies the compatibility
condition $\int_{\pa\Om}\eta\d S=0$ and $\psi\in C^1(\BR_+)$ satisfies
\begin{align}\label{eq:g-2}
\psi(t)=\begin{cases}
0, & t<T_0,\\
1, & t>T_1,
\end{cases}
\end{align}
with $0< T_0<T_1<T$. The measurement data $h(x,t)=u(x,t)$ is taken on a part of the
boundary $\Ga_0\subset\pa\Om$. Note that the inverse problem involves
missing data ($u_0$ and $f$), whereas the available data is only on a partial
boundary. Thus, it is both mathematically and numerically very challenging, due to not
only the severe ill-posed nature and high degree of nonlinearity
but also the unknown forward map from the parameters $a$ and $\al$ to the data
$h(x,t)$.

The mathematical study on inverse problems for time-fractional models is of relatively recent
origin, starting from the pioneering work \cite{ChengNakagawa:2009} (see \cite{JinRundell:2015,
LiLiuYamamoto:2019b,LiYamamoto:2019a} for overviews) and there are several
existing works on recovering a space-dependent potential or diffusion coefficient from lateral
Cauchy data \cite{RundellYamamoto:2018,RundellYamamoto:2020,WeiYan:2021,JingYamamoto:2021,JinZhou:2021ip,Kian:2022}.
Rundell and Yamamoto \cite{RundellYamamoto:2018} showed that the lateral Cauchy data can uniquely
determine the spectral data when $u_0\equiv f\equiv0$, and proved the unique determination
of the potential using Gel'fand-Levitan theory. They also numerically studied the singular value spectrum of
the linearized forward map, showing the severe ill-posed nature of the problem. Later,
they \cite{RundellYamamoto:2020} relaxed the regularity condition on the boundary excitation
$g(t)$ in a suitable Sobolev space. Recently, Jing and Yamamoto \cite{JingYamamoto:2021}
proved the identifiability of multiple parameters (including order, spatially dependent potential,
initial value and Robin coefficients in the boundary condition) in a
time-fractional subdiffusion model with a zero boundary condition and source,
excited by a nontrivial initial condition from the lateral Cauchy data at both end points; see also \cite{JingPeng:2020}. Jin and Zhou
\cite{JinZhou:2021ip} studied the unique recovery of the potential, fractional order and either initial
data or source from the lateral Cauchy data, when the boundary excitation is judiciously chosen.
All these interesting works are concerned with the one-dimensional setting due to their
essential use of the inverse Sturm-Liouville theory. Wei et al \cite{WeiZhang:2023} numerically
investigated the recovery of the zeroth-order coefficient and fractional
order in a time-fractional reaction-diffusion-wave equation from lateral boundary data. A direct extension of these theoretical
works to the multi-dimensional case is challenging since the Gel'fand-Levitan theory
is no longer applicable. Kian et al. \cite{KianLiLiuYamamoto:2021}
provided the first results for the multi-dimensional case, including the uniqueness for identifying
two spatially distributed parameters in the subdiffusion model from one
single lateral observation with a specially designed excitation {Dirichlet input}; see also
\cite{HelinLassasZhang:2020} for a related study on determining the manifold from one measurement
corresponding to a specialized source. Kian
\cite{Kian:2022} studied also the issue of simultaneous recovery of these parameters along with the
order and initial data using a similar choice of the boundary data. However, in the
works \cite{KianLiLiuYamamoto:2021,Kian:2022}, the excitation data, which plays the role of infinity measurements,
is numerically inconvenient to realize, if not impossible at all; see Remark \ref{rem:general-a}
and the appendix for further discussions. These considerations motivate the
current work, i.e., to design robust numerical algorithm for recovering multiple parameters
from a single partial boundary measurement for multi-dimensional subdiffusion with a computable
excitation Neumann data, in the presence of a partly unknown medium.

In this work, we make the following contributions to the mathematical analysis and numerics of
the concerned inverse problem. First, we examine the feasibility to
recover multiple parameters. We show that if the coefficient $a$ is
piecewise constant as defined  in \eqref{eqn:a}, then one single boundary measurement can uniquely
determine the coefficient $a$ and fractional order $\al$, even though the initial data $u_0$ and
source $f$ are unknown. Note that the exciting Neumann data $g$ given in \eqref{eq:g} is easy to
realize and hence allows the numerical recovery. The proof relies on the asymptotic behavior of
Mittag-Leffler functions, analyticity in time of the solution, and the uniqueness of the inverse
conductivity problem (for elliptic problems) from one boundary measurement. In particular, the
subdomain $D$ can be either a convex polygon / polyhedron or a disc / ball, cf. Theorem \ref{thm:unique q}
and Remark \ref{rem:poly-ball}. This analysis strategy follows a well-established
procedure in the community, and roughly it consists of two steps. (1) Using the time-analyticity,
the uniqueness for the original inverse problem is reduced to the
one for an inverse problem for the corresponding time-independent elliptic equation; (2) The
reduction can be done by the Laplace transform or considering the asymptotics. Both strategies of reductions
are well known. For example, the former way is used for an Dirichlet-to-Neumann map for the inverse
coefficient problem for a multi-term time-fractional diffusion equation \cite{LiImanuvilovYamamoto:2016},
while the latter way is used for the Dirichlet-to-Neumann map for the inverse parabolic problem
\cite[Section 4, Chapter 9]{isakov2017inverse}. Second, the uniqueness analysis lends itself to the development
of a robust numerical algorithm: we develop a three-step recovery algorithm for identifying the
piecewise constant coefficient $a$ and the fractional order $\al$:
 (i) use the asymptotic behavior of the solution of problem \eqref{eqn:fde} near $t=0$ to recover $\al $;
 (ii) use analytic continuation to extract the solution of problem \eqref{eqn:fde} with zero $f$ and $u_0$;
 (iii) use the level set method to recover the shape of subdomain $D$.
To the best of our knowledge, this is the first work on the numerical recovery of a (piecewise constant)
diffusion coefficient in the context of multi-dimensional subdiffusion model with missing initial and source data.
Last, we present extensive numerical experiments to illustrate the feasibility of the approach.
We refer interested readers to  \cite{RundellZhang:2018jcp, PrakashHriziNovotny:2022} for some numerical studies
for identifying a piecewise constant source from the boundary measurement.

The rest of the paper is organized as follows. In Section \ref{sec:prelim} we describe preliminary
results on the model, especially time analyticity of the data. Then in Section
\ref{sec:uniqueness} we give the uniqueness result in case of
piecewise constant $a$, and in Section \ref{sec:alg} we develop a recovery algorithm based
on the level set method. We present extensive numerical experiments to illustrate the feasibility
of recovering multiple parameters in Section \ref{sec:numer}. In an appendix, we discuss the possibility of recovering
two coefficients from one boundary measurement induced by a specialized boundary excitation.
Throughout, the notation $(\,\cdot\,,\,\cdot\,)$ denotes the standard $L^2(\Om)$ inner product, and
$\langle\,\cdot\,,\,\cdot\,\rangle$ the $L^2(\pa\Om)$ inner product. For a Banach space $B$,
$C^\om(T,\infty;B)$ denotes the set of functions valued in $B$ and analytic in $t\in(T,\infty)$.
The notation $c$, with or without a subscript, denotes a generic constant which may change at each
occurrence, but it is always independent of the concerned quantities.

\section{Preliminaries}\label{sec:prelim}
In this section, we present preliminary analytical results. Let $A$ be the $L^2(\Om)$
realization of the elliptic operator $\cA$, with a domain $\mathrm{Dom}(A):=\{v\in
L^2(\Om):\cA v\in L^2(\Om), \pa_\nu v|_{\pa\Om}=0\}$.
Let $\{\la_\ell\}_{\ell\ge1}$ be a strictly increasing sequence of eigenvalues of
$A$, and denote the multiplicity of $\la_\ell$ by $m_\ell$ and
$\{\vp_{\ell,k}\}_{k=1}^{m_\ell}$ an $L^2(\Om)$ orthonormal basis of
$\ker(A-\la_\ell)$. That is, for any $\ell\in\BN$, $k=1,\dots,m_\ell$:
\begin{align}
\begin{cases}
\cA\vp_{\ell,k}=\la_\ell\vp_{\ell,k} & \mbox{in }\Om,\\
a\pa_\nu\vp_{\ell,k}=0 & \mbox{on }\pa\Om.
\end{cases}
\end{align}
The eigenvalues $\{\la_\ell\}_{\ell=1}^\infty$ are nonnegative, and the eigenfunctions
$\{\vp_{\ell,k}:k=1,\dots,m_\ell\}_{\ell=1}^\infty$ form {a complete} orthonormal basis
of $L^2(\Om)$. Note that $\la_1=0$ (and has multiplicity $1$)
and the corresponding eigenfunction $\vp_1=|\Om|^{-\frac12}$
is constant valued, where $|E|$ denotes the Lebesgue measure of a set $E$. Due to the
piecewise constancy of the coefficient $a$, $\vp_{\ell,k}$ is smooth
in $D$ and $\Om\setminus\ov D$. Moreover, it
satisfies the following transmission condition on the interface $\pa D$:
\begin{align}\label{eqn:trans-cond}
\vp_{\ell,k}|_-=\vp_{\ell,k}|_+\quad\mbox{and}\quad\pa_n\vp_{\ell,k}|_-=(1+\mu)\pa_n\vp_{\ell,k}|_+\quad\mbox{on }\pa D,
\end{align}
where $\vp_{\ell,k}|_+$ and $\vp_{\ell,k}|_-$ denote the limits from $D$ and $\Om\setminus\ov D$
to the interface $\pa D$, respectively, and $\pa_n\vp_{\ell,k}|_\pm$ denotes
the derivative with respect to the unit outer normal vector $n$ on $\pa D$.
Then we define the fractional power $A^s$ ($s\ge0$) via functional calculus by
\[
A^s v:=\sum_{\ell=1}^\infty\la_\ell^s\sum_{k=1}^{m_\ell}(v,\vp_{\ell,k})\vp_{\ell,k},
\]
with a domain $\mathrm{Dom}(A^s)=\{v\in L^2(\Om):A^s v\in L^2(\Om)\}$, and the associated graph norm
$$\|v\|_{{\rm Dom}(A^s)}=\Big(\sum_{\ell=1}^\infty \la_\ell^{2s}\sum_{k=1}^{m_\ell}(v,\vp_{\ell,k})^2\Big)^\frac{1}{2}.$$

We use extensively the Mittag-Leffler function $E_{\al,\be}(z)$ defined by
(\cite[pp. 40-45]{KilbasSrivastavaTrujillo:2006}, \cite[Section 3.1]{Jin:book2021})
\[
E_{\al,\be}(z)=\sum_{k=0}^\infty\f{z^{k}}{\Ga(k\al+\be)},\quad\forall z\in\BC.
\]
The function $E_{\al,\be}(z)$ generalizes the exponential function
$\e^z$. The following decay estimate of $E_{\al,\be}(z)$ is crucial in the analysis
below; See e.g., \cite[eq. (1.8.28), p. 43]{KilbasSrivastavaTrujillo:2006} and
\cite[Theorem 3.2]{Jin:book2021} for the proof.

\begin{lemma}\label{lem:ML-asymp}
Let $\al\in(0,2)$, $\be\in\BR$, $\vp\in(\frac{\al}{2}\pi,\min(\pi,\al\pi))$ and $N\in\BN$.
Then for $\varphi\le|\arg z|\le\pi$ with $|z|\to\infty$, there holds
$$
E_{\al,\be}(z)=-\sum_{k=1}^N\f{z^{-k}}{\Ga(\be-\al k)}+O(|z|^{-N-1}).
$$
\end{lemma}

By linearity, we may split the solution $u$ of problem \eqref{eqn:fde} into $u=u_i+u_b$, with $u_i$ and $u_b$ solving
\begin{align}\label{eqn:PDE-uib}
\begin{cases}
\pa_t^\al u_i+\cA u_i=f & \mbox{in }\Om\times(0,T],\\
a\pa_\nu u_i=0 & \mbox{on }\pa\Om\times(0,T],\\
u_i(0)=u_0 & \mbox{in }\Om
\end{cases}\quad\mbox{and}\quad\begin{cases}
\pa_t^\al u_b+\cA u_b=0 & \mbox{in }\Om\times(0,T],\\
a\pa_\nu u_b=g & \mbox{on }\pa\Om\times(0,T],\\
u_b(0)=0 & \mbox{in }\Om,
\end{cases}
\end{align}
respectively. The following result gives the representations of $u_i$ and $u_b$.

\begin{proposition}\label{reps of ui}
Let $u_0,f\in L^2(\Om)$. Then there exist unique solutions
$u_i,u_b\in L^2(0,T;H^1(\Om))$ that can be respectively represented by
\begin{align*}
u_i(t) & =(u_0,\vp_1){\vp_1}+\f{(f,\vp_1){\vp_1}t^\al}{\Ga(1+\al)}\\
 &\quad +\sum_{\ell=2}^\infty\sum_{k=1}^{m_\ell}\left(\left[(u_0,\vp_{\ell,k})-\la_\ell^{-1}(f,\vp_{\ell,k})\right]E_{\al,1}(-\la_\ell t^\al)+\la_\ell^{-1}(f,\vp_{\ell,k})\right)\vp_{\ell,k},\\
u_b(t) & =\sum_{\ell=1}^\infty\sum_{k=1}^{m_\ell}\int_0^t(t-s)^{\al-1}E_{\al,\al}(-\la_{\ell,k}(t-s)^\al)\langle g(s),\vp_{\ell,k}\rangle\,\d s\,\vp_{\ell,k}.
\end{align*}
Hence, the solution $u$ to problem \eqref{eqn:fde} can be represented as
\[
u(t)=\rho_0+\rho_1t^\al+\sum_{\ell=2}^\infty E_{\al,1}(-\la_\ell t^\al)\rho_\ell+\sum_{\ell=1}^\infty\int_0^t(t-s)^{\al-1}E_{\al,\al}(-\la_\ell(t-s)^\al)\sum_{k=1}^{m_\ell}\langle g(s),\vp_{\ell,k}\rangle\,\d s\,\vp_{\ell,k},
\]
with $\rho_\ell$ given by
\begin{align}\label{eqn:varrho}
\rho_\ell:=\left\{\!\begin{alignedat}{2}
& (u_0,\vp_1)\vp_1+\sum_{\ell=2}^\infty\sum_{k=1}^{m_\ell}\la_\ell^{-1}(f,\vp_{\ell,k})\vp_{\ell,k}, & \quad &\ell=0,\\
&\f{(f,\vp_1)}{\Ga(1+\al)}\vp_1, & \quad &\ell=1,\\
&\sum_{k=1}^{m_\ell}\left[(u_0,\vp_{\ell,k})-\la_\ell^{-1}(f,\vp_{\ell,k})\right]\vp_{\ell,k}, & \quad &\ell=2,3,\dots.
\end{alignedat}\right.
\end{align}
\end{proposition}

\begin{proof}
The representations follow from the standard separation of variables technique
(\cite{SakamotoYamamoto:2011}, \cite[Section 6.2]{Jin:book2021}).
The piecewise constancy of the diffusivity $a$ requires special
care due to a lack of global regularity. By multiplying
the governing equation of $u_i$ by $\vp_{\ell,k}$ and then integrating over
 $\Om$, we get
\[
\pa_t^\al(u_i(t),\vp_{\ell,k})+(\cA u_i(t),\vp_{\ell,k})=(f,\vp_{\ell,k}).
\]
Integrating by parts twice and using the transmission condition \eqref{eqn:trans-cond}
for $\vp_{\ell,k}$ (and $u_i$) on $\pa D$ gives
\begin{align*}
(\cA u_i(t),\vp_{\ell,k}) & =-\int_{\Om\setminus\ov D}\nb\cdot (\nb u_i)\,\vp_{\ell,k}\,\d x-\int_D\nb\cdot ((1+\mu )\nb u_i)\,\vp_{\ell,k}\,\d x\\
& =-\int_{\pa\Om}(\nb u_i\cdot\nu)\,\vp_{\ell,k}\,\d S-\int_{\pa D}(\nb u_i\cdot n_-)\,\vp_{\ell,k}|_-\,\d S+\int_{\Om\setminus\ov D}\nb u_i\cdot\nb\vp_{\ell,k}\,\d x\\
& \quad\,-\int_{\pa D}(1+\mu)(\nb u_i\cdot n_+)\,\vp_{\ell,k}|_+\,\d S+\int_D(1+\mu)\nb u_i\cdot\nb\vp_{\ell,k}\,\d x\\
& =\int_{\Om\setminus\ov D}\nb u_i\cdot\nb\vp_{\ell,k}\,\d x+\int_D(1+\mu )\nb u_i\cdot\nb\vp_{\ell,k}\,\d x\\
& =\int_{\pa\Om}(\nb\vp_{\ell,k}\cdot\nu)\,u_i\,\d S+\int_{\pa D}(\nb\vp_{\ell,k}\cdot n_-)\,u_i|_-\,\d S-\int_{\Om\setminus\ov D}\nb\cdot(\nb\vp_{\ell,k})\,u_i\,\d x\\
& \quad\,+\int_{\pa D}(1+\mu)(\nb\vp_{\ell,k}\cdot n_+)\,u_i|_+\,\d S-\int_D \nb\cdot((1+\mu)\nb\vp_{\ell,k})\,u_i\,\d x\\
& =(u_i,\cA\vp_{\ell,k})=\la_\ell(u_i,\vp_{\ell,k}).
\end{align*}
Hence, the scalar function $u_i^{\ell,k}(t):=(u_i(t),\vp_{\ell,k})$ satisfies the following initial value problem for a fractional ordinary differential equation:
\[
(\pa_t^\al+\la_\ell)u_i^{\ell,k}(t)=f_{\ell,k}:=(f,\vp_{\ell,k})\qquad\text{for}~ 0<t\le T,\quad\text{with}~
u_i^{\ell,k}(0)=u_0^{\ell,k}:=(u_0,\vp_{\ell,k}).
\]
Then $u_i^{\ell,k}(t)$ is given by \cite[Proposition 4.5]{Jin:book2021}
\[
u_i^{\ell,k}(t)=u_0^{\ell,k}E_{\al,1}(-\la_\ell t^\al)+f_{\ell,k}\int_0^t s^{\al-1}E_{\al,\al}(-\la_\ell s^\al)\,\d s.
\]
Note that $u_i^1=u_0^1+\f1{\Ga(1+\al)}f_1t^\al$.
Now using the identity
\begin{align}\label{eqn:mlf-diff}
\f\d{\d t}E_{\al,1}(-\la t^\al)=-\la\, t^{\al-1}E_{\al,\al}(-\la t^\al),
\end{align}
we have for $\ell\ge2$ and $k=1,\dots,m_\ell$ that
\begin{align*}
u_i^{\ell,k}(t)&=u_0^{\ell,k}E_{\al,1}(-\la_\ell t^\al)+\la_\ell^{-1}\left[1-E_{\al,1}(-\la_\ell t^\al)\right]f_{\ell,k}\\
&=\left(u_0^{\ell,k}-\la_\ell^{-1}f_{\ell,k}\right)E_{\al,1}(-\la_\ell t^\al)+\la_\ell^{-1}f_{\ell,k}.
\end{align*}
This gives the representation of $u_i$. Similarly, multiplying
the governing equation for $u_b$ by $\vp_{\ell,k}$ and integrating over $\Om$ give
$\pa_t^\al(u_b(t),\vp_{\ell,k})+(\cA u_b(t),\vp_{\ell,k})=0$.
Repeating the argument yields that {$u_b^{\ell,k}(t):=(u_b(t),\vp_{\ell,k})$ satisfies
\[
(\pa_t^\al+\la_\ell)u_b^{\ell,k}(t)=\langle g(t),\vp_{\ell,k}\rangle\quad \text{for}~0<t\le T,\qquad \text{with}~~
u_b^{\ell,k}(0)=0.
\]
The solution $u_b^{\ell,k}(t)$} is given by \cite[Proposition 4.5]{Jin:book2021}
$$
u_b^{\ell,k}(t)=\int_0^t(t-s)^{\al-1}E_{\al,\al}(-\la_\ell(t-s)^\al)\langle g(s),\vp_{\ell,k}\rangle\,\d s\,\vp_{\ell,k}.
$$
Thus the desired assertion follows. The representation of the solution $u$ to problem
\eqref{eqn:fde} follows directly from that of $u_b$ and $u_i$, and the identity \eqref{eqn:mlf-diff}.
\end{proof}

Next we show properties of the boundary data $h$. This is achieved by first
proving related properties of $u$ and then applying the trace theorem. Below we study the
analyticity of
\begin{align*}
u_i{ (t)}&=\rho_0+\rho_1t^\al+\sum_{\ell=2}^\infty E_{\al,1}(-\la_\ell t^\al)\rho_\ell,\\
u_b{ (t)}&=\sum_{\ell=1}^\infty\sum_{k=1}^{m_\ell}\int_0^t(t-s)^{\al-1}E_{\al,\al}(-\la_\ell(t-s)^\al)\langle g(s),\vp_{\ell,k}\rangle\,\d s\,\vp_{\ell,k}.
\end{align*}
Since our focus is the trace on $\pa\Om$, we only study $u$ on the subdomain $\Om\setminus\ov D$.
Recall that for a Banach space $B$, the notation $C^\om(T,\infty;B)$ denotes the set
of functions valued in $B$ and analytic in $t\in(T,\infty)$.
\begin{proposition}\label{analytic of u}
Let $D'\supset D$ be a small neighborhood of $D$ with a smooth boundary and denote $\Om'=\Om\setminus\ov{D'}$.  For $u_0\in L^2(\Om)$, $f\in L^2(\Om)$ and $g$ as in \eqref{eq:g},  the following statements hold.
\begin{enumerate}
\item[(i)] 
$u_i\in C^\om(0,\infty;H^2(\Om'))$ and $u_b\in C^\om(T_1+\ve,\infty;H^2(\Om'))$ {for arbitrarily fixed $\ve>0$}.
\item[(ii)] The Laplace transforms {$\wh u_i(z)$ and $\wh u_b(z)$ of $u_i$ and $u_b$ in $t$} exist for all $ \Re(z)>0$
 and are respectively given by
\[
\wh u_i(z)=z^{-1}\rho_0+\Ga(\al+1)z^{-\al-1}\rho_1+\sum_{\ell=2}^\infty\f{\rho_\ell z^{\al-1}}{z^\al +\la_\ell}\quad\mbox{and}\quad\wh u_b(z)=\sum_{\ell=1}^\infty\sum_{k=1}^{m_\ell}\f{\langle\wh g(z),\vp_{\ell,k}\rangle\vp_{\ell,k}}{z^\al+\la_\ell}.
\]
\end{enumerate}
\end{proposition}
\begin{proof}
Throughout this proof, let $\ve>0$ be arbitrarily fixed.
Since $\la_1=0$, by Lemma \ref{lem:ML-asymp}, there exist constants $c>0$
and $\te\in(0,\frac{\pi}{2})$ such that for any $z\in\Si_\te:=\{z\in\BC\setminus\{0\}:|\arg(z)|\le\te\}$, we have
\begin{align*}
&\quad \|u_i(z)\|_{{\rm Dom}(A)}^2 =\sum_{n=1}^\infty\la_n^2\sum_{j=1}^{m_n}\left(\vp_{n,j},\rho_0+\rho_1 z^\al+\sum_{\ell=2}^\infty E_{\al,1}(-\la_\ell z^\al)\rho_\ell\right)^2\\
& =\sum_{n=2}^\infty\la_n^2\sum_{j=1}^{m_n}\left(\vp_{n,j},\sum_{\ell=2}^\infty\sum_{k=1}^{m_\ell}\left\{E_{\al,1}(-\la_\ell z^\al)\left[(u_0,\vp_{\ell,k})-\la_\ell^{-1}(f,\vp_{\ell,k})\right]{+\la_\ell^{-1}(f,\vp_{\ell,k})}\right\}\vp_{\ell,k}\right)^2\\
& \le c\sum_{n=2}^\infty\la_n^2E_{\al,1}(-\la_n z^\al)^2\sum_{j=1}^{m_n}\left\{(u_0,\vp_{n,j})^2+\la_n^{-2}(f,\vp_{n,j})^2\right\}{+c\sum_{n=1}^\infty\sum_{j=1}^{m_n}(f,\vp_{n,j})^2}\\
& \le c|z|^{-2\al}\sum_{n=2}^\infty\sum_{j=1}^{m_n}\left\{(u_0,\vp_{n,j})^2+\la_n^{-2}(f,\vp_{n,j})^2\right\}{+ c\|f\|_{L^2(\Om)}^2}\\
& \le c|z|^{-2\al}\big(\| u_0\|_{L^2(\Om)}^2+\| f\|_{L^2(\Om)}^2\big){+c\|f\|_{L^2(\Om)}^2}.
\end{align*}
Since $u_0,f\in L^2(\Om)$, $\| u_i(z)\|_{\mathrm{Dom}(A)}^2$ is uniformly bounded for $z\in\Si_\te$. Since $E_{\al,1}(-\la_n z^\al)$ is analytic in $z\in\Si_\te$ and the series converges  uniformly in any compact subset of $\Si_\te$, $u_i(t)$ is analytic in $t\in(0,\infty)$ as a ${\rm Dom}(A)$-valued function, i.e., $u_i\in C^\om(0,\infty;\mathrm{Dom}(A))$. By Sobolev embedding, $u_i\in C^\om(0,\infty;H^2(\Om'))$.

Next we prove the analyticity of $u_b$. By the choice $g(x,t)=\eta(x)\psi(t)$ in \eqref{eq:g}
and integration by parts, for $t>T_1$, $u_b^1(t):=(u_b(t),\vp_1)$ is given by
\begin{align*}
	u_b^1(t)&=\f1{\Ga(\al)}\int_0^t(t-s)^{\al-1}\langle g(s),\vp_1\rangle\,\d s=
     \f{\langle \eta,\vp_1\rangle}{\Ga(\al)}\int_0^t(t-s)^{\al-1} \psi(s)\,\d s \\
     &= \f{\langle \eta,\vp_1\rangle}{\al\,\Ga(\al)}\left[ -(t-s)^{\al}\psi(s)|_{s=0}^{s=t} + \int_0^t(t-s)^{\al}\psi'(s)\,\d s\right]\\
	&=\f{\langle\eta,\vp_1\rangle}{\Ga(\al+1)}\int_{T_0}^{T_1}(t-s)^\al\psi'(s)\,\d s,
\end{align*}
where the last step follows from the condition on $\psi$ in \eqref{eq:g-2}. Thus the time-analyticity of $u_b^1(t)\vp_1$ for $t\in(T_1+\ve,\infty)$ follows. Next, again by
integration by parts, \eqref{eq:g}--\eqref{eq:g-2} and the identity \eqref{eqn:mlf-diff}, for $t>T_1$, $u_b^{\ell,k}(t):=(u_b(t),\vp_{\ell,k})$ with $\ell
\ge2$, $k=1,\dots,m_\ell$ can be written as
\begin{align*}
u_b^{\ell,k}(t) & =\int_0^t(t-s)^{\al-1}E_{\al,\al}(-\la_\ell(t-s)^\al)\langle g(s),\vp_{\ell,k}\rangle\,\d s\\
&=\int_0^t\f{\langle g(s),\vp_{\ell,k}\rangle}{\la_\ell}\f\d{\d s}E_{\al,1}(-\la_\ell(t-s)^\al)\,\d s\\
& =\la_\ell^{-1}\Big[\langle g(s),\vp_{\ell,k}\rangle E_{\al,1}(-\la_\ell(t-s)^\al)\Big]_{s=0}^{s=t}-\f{\langle\eta,\vp_{\ell,k}\rangle}{\la_\ell}\int_0^tE_{\al,1}(-\la_\ell(t-s)^\al)\psi'(s)\,\d s\\
& =\f{\langle\eta,\vp_{\ell,k}\rangle}{\la_\ell}\psi(t)-\f{\langle\eta,\vp_{\ell,k}\rangle}{\la_\ell}\int_{T_0}^{T_1}E_{\al,1}(-\la_\ell (t-s)^\al)\psi'(s)\,\d s=:u_{b,1}^{\ell,k}(t)+u_{b,2}^{\ell,k}(t).
\end{align*}
Since $\psi(t)=1$ for $t>T_1$, we see that $u_{b,1}^{\ell,k}(t)$ is a constant for $t>T_1$.  Next we consider the following boundary value problem
\begin{align}\label{eq:U}
\cA U=0~~  \mbox{in }\Om,\quad \text{with}\quad
a\pa_\nu U=\eta  ~~\mbox{on }\pa\Om.
\end{align}
The compatibility condition $\langle\eta,1\rangle=0$ implies that there exist solutions to problem \eqref{eq:U}.
We take an arbitrary solution $U$. Since $a$ is piecewise constant and $\eta\in H^{\frac{1}{2}}(\pa\Om)$, we know that $U\in H^1(\Om)$ and its restriction $U|_{\Om'}\in H^2(\Om')$. Integrating by parts twice yields
$$\langle\eta,\vp_{\ell,k}\rangle=\la_\ell(U,\vp_{\ell,k}) .$$
Similar to the argument for Proposition \ref{reps of ui}, from the transmission condition \eqref{eqn:trans-cond}, we deduce
 \[
	\sum_{\ell=2}^\infty\sum_{k=1}^{m_\ell}u_{b,1}^{\ell,k}(t)\vp_{\ell,k}=\sum_{\ell=2}^\infty\sum_{k=1}^{m_\ell}\psi(t)(U,\vp_{\ell,k})\vp_{\ell,k} ,
	\]
which is analytic in $t\in(T_1+\ve,\infty)$ since it is constant in time and
$U\in L^2(\Om)$. Moreover, by the standard elliptic regularity theory,
$$\sum_{\ell
=2}^\infty\sum_{k=1}^{m_\ell}u_{b,1}^{\ell,k}\vp_{\ell,k}\in C^\om(T_1+\ve,\infty;H^2(\Om')).$$
Recall Young's inequality for convolution, i.e., $\|f*g\|_{L^r(\BR)}
\le\|f\|_{L^p(\BR)}\|g\|_{L^q(\BR)}$ for $p,q,r\ge1$ with $p^{-1}+
q^{-1}=r^{-1}+1$ and any $f\in L^p(\BR)$ and $g\in L^q(\BR)$. Then
by Young's inequality, Lemma \ref{lem:ML-asymp} and the regularity estimate $\sum_{\ell=2}^\infty\la_\ell^{-2}\sum_{k=1}^{m_\ell}\langle\eta,\vp_{\ell,k}\rangle^2\leq\|U\|_{L^2(\Om)}<\infty$, we deduce
\begin{align*}
\left\|\sum_{\ell=2}^\infty\sum_{k=1}^{m_\ell}u_{b,2}^{\ell,k}(z)\vp_{\ell,k}\right\|_{{\rm Dom}(A)}^2 & =\sum_{n=1}^\infty\la_n^2\sum_{j=1}^{m_n}\left(\vp_{n,j},\sum_{\ell=2}^\infty\sum_{k=1}^{m_\ell}u_{b,2}^{\ell,k}(z)\vp_{\ell,k}\right)^2\\
& ={\sum_{n=2}^\infty\la_n^2\sum_{j=1}^{m_n}\left(\f{\langle\eta,\vp_{n,j}\rangle}{\la_n}\int_{T_0}^{T_1}E_{\al,1}(-\la_n(z-s)^\al)\psi'(s)\,\d s\right)^2}\\
& \le{\sum_{n=2}^\infty\sum_{j=1}^{m_n}\langle\eta,\vp_{n,j}\rangle^2\left(\f{c}{\la_n|z-T_1|^\al}\int_{T_0}^{T_1}|\psi'(s)|\,\d s\right)^2}\\
& \le{\left(\f{c\|\psi\|_{W^{1,\infty}(\BR_+)}}{|z-T_1|^\al}\right)^2\sum_{n=2}^\infty\la_n^{-2}\sum_{j=1}^{m_n}|\langle\eta,\vp_{n,j}\rangle|^2}\le\f{c}{|z-T_1|^{2\al}}.
\end{align*}
Since $u_{b,2}^{\ell,k}(t)$ is analytic in $(T_1+\ve,\infty) $ and the series
$\sum_{\ell=2}^\infty\sum_{k=1}^{m_\ell}u_{b,2}^{\ell,k}(z)\vp_{\ell,k}$ converges
uniformly in $\mathrm{Dom}(A)$ for $z\in T_1+\ve+\Si_\te$, it belongs to $C^\om(T_1+\ve,\infty;\mathrm{Dom}(A)) $, and hence $u_b\in C^\om(T_1+\ve,\infty;H^2(\Om'))$.  This proves part (i).

The argument for part (i) implies that the series converges uniformly in $\mathrm{Dom}(A)$ for $t\in(0,\infty)$, and
\[
\|\e^{-t z}u_i(t)\|_{{\rm Dom}(A)}\le c\,\e^{-t\,\Re(z)}{(t^{-\al}+1)},\quad t>0.
\]
The function $\e^{-t\,\Re(z)}(t^{-\al}+1)$ is integrable in $t$ over $(0,\infty)$ for any fixed $z$ with $\Re(z)>0$. By Lebesgue's dominated convergence theorem and taking Laplace transform termwise, we obtain
\[
\wh u_i(z)=z^{-1}\rho_0+\Ga(\al+1)z^{-\al-1}\rho_1+\sum_{\ell=2}^\infty\f{\rho_\ell z^{\al-1}}{z^\al+\la_\ell},\quad\forall\,\Re(z)>0.
\]
The argument for part (i) also implies
\[
\|\e^{-t z}u_b(t)\|_{{\rm Dom}(A) }\le c\,\e^{-t\,\Re(z)}|t-T_1|^{-\al},\quad t>0.
\]
Then termwise Laplace transform and Lebesgue's dominated convergence theorem complete the proof of the proposition.
\end{proof}

Thus, $u_i$ and $u_b$ are analytic in time and have
$H^2(\Om')$ regularity. Since $\pa\Om$ is Lipschitz and piecewise $C^{1,1}$,
their traces on $\pa\Om$ are well defined.  The next result is
direct from the trace theorem and Sobolev embedding theorem. Here, we use $x$ and $y$
denote the variables in $\Om$ and on $\pa\Om$,
respectively.

\begin{corollary}\label{reps of h}\label{analytic of h}
Let the assumptions in Proposition \ref{analytic of u} hold. Then the data $h=u|_{\Ga_0\times(0,T)}$ to problem \eqref{eqn:fde} can be represented by
\begin{align*}
h(t)=&\underbrace{\rho_0+\rho_1t^\al+\sum_{\ell=2}^\infty E_{\al,1}(-\la_\ell t^\al)\rho_\ell}_{=:h_i(t)}\\
  &+\underbrace{\sum_{\ell=1}^\infty\sum_{k=1}^{m_\ell}\int_0^t(t-s)^{\al-1}E_{\al,\al}(-\la_\ell(t-s)^\al)\langle g(s),\vp_{\ell,k}\rangle\,\d s\,\vp_{\ell,k}}_{=:h_b(t)}.
\end{align*}
Moreover, $h_i$ and $h_b$ satisfy the following properties.
\begin{enumerate}
\item[(i)] $h_i\in C^\om(0,\infty;L^2(\Ga_0))$ and $h_b\in C^\om(T_1+\ve,\infty;L^2(\Ga_0))$ for arbitrarily fixed $\ve>0$.
\item[(ii)] The Laplace transforms $\wh h_i(z)$ and $\wh h_b(z)$ of $h_i$ and $h_b$ in $t$ exist for all $ \Re(z)>0$ and are given by
\begin{align*}
\wh h_i(z)&=z^{-1}\rho_0+\Ga(\al+1)z^{-\al-1}\rho_1+\sum_{\ell=2}^\infty\f{\rho_\ell z^{\al-1}}{z^\al+\la_\ell},\\
\wh h_b(z)&=\sum_{\ell=1}^\infty\sum_{k=1}^{m_\ell}\f{\langle\wh g(z),\vp_{\ell,k}\rangle\vp_{\ell,k}}{z^\al+\la_\ell}.
\end{align*}
\end{enumerate}
\end{corollary}
\begin{remark}
The analysis of Theorem \ref{thm:unique al} crucially exploits the analyticity of the measurement $h_i(t)$ in time, which relies on
condition \eqref{eq:g-2}, i.e., $\psi(t)\equiv 0$ for $t\in [0,T_0]$. The condition $\psi(t)\equiv 1$ for $t\geq T_1$ for some $T_1<T$ from \eqref{eq:g-2} ensures the time analyticity of $h_b(t)$ for $t>T_1+\varepsilon$, which is needed for Theorem \ref{thm:unique q}. It should be interpreted as analytically extending the observation $h_b(t)$ by analytically extending $\psi(t)$, both  from $(T_1,T)$ to $(T_1,\infty)$. Alternative conditions on $\psi(t)$ ensuring the time analyticity of $h_b(t)$ for $t>T_1+\varepsilon$, e.g., $\psi(t)$ vanishes identically on $(T_1,T)$, would also be sufficient for Theorem \ref{thm:unique q}.
\end{remark}

\section{Uniqueness}\label{sec:uniqueness}
Now we establish a uniqueness result for recovering the fractional order $\al$ and piecewise constant $a$.
The proof proceeds in two steps: First we show the uniqueness of the order $\al$ from the observation, despite
that the initial condition $u_0$ and source $f$ are unknown. Then we show the uniqueness of $a$. The key
observation is that the contributions from $u_0$ and $f$ can be extracted explicitly.
Since the Dirichlet data is only available on a sub-boundary $\Ga_0$, we view $\rho_k$ as a $L^2(\Ga_0)$-valued
function. The notation $\BK$ denotes the set $\{k\in\BN:\rho_k\not\equiv0\mbox{ in }L^2(\Ga_0)\}$, i.e., the support
of the sequence $(\rho_0,\rho_1,\dots)$ in $L^2(\Ga_0)$ sense, similarly, $\wt\BK=\{k\in\BN:\wt\rho_k\not\equiv0
\mbox{ in }L^2(\Ga_0)\}$, and $\BN^*=\BN\setminus\{1\}$. Below we denote by $\mathfrak A$ the admissible set
of conductivities, i.e.,
\begin{equation*}
\mathfrak{A} = \{1+\mu\chi_D(x): \mu>-1\mbox{ and } D\subset \Omega \mbox{ is  a convex polygon}\}.
\end{equation*}

\begin{theorem}\label{thm:unique al}
Let $\al,\wt\al\in(0,1)$, $(a,f,u_0),(\wt a,\wt f,\wt u_0)\in\mathfrak A\times L^2(\Om)\times L^2(\Om)$, and fix $g$ as \eqref{eq:g} with $\psi(t)$ satisfying condition \eqref{eq:g-2}.  Let $h$ and $\wt h$ be the corresponding Dirichlet observations. Then for some $\si>0$, the condition $h=\wt h$ on $\Ga_0\times[T_0-\si,T_0]$ implies $\al=\wt\al$, $\rho_0=\wt\rho_0$ and $\{(\rho_k,\la_k)\}_{k\in\BK}
=\{(\wt\rho_k,\wt\la_k)\}_{k\in\wt\BK}$ if $\BK,\wt\BK\ne\emptyset$.
\end{theorem}
\begin{proof}
By the definition of $g$, we have $g(y,t)\equiv0$ for $y\in\pa\Om$, $t\in[0,T_0]$. Then by Corollary \ref{reps of h}, $h(y,t)$ admits a Dirichlet representation
\[
h(y,t)=\rho_0(y)+\rho_1(y)t^\al+\sum_{k\in\BK\cap\BN^*}\rho_k(y)E_{\al,1}(-\la_k t^\al).
\]
By Corollary \ref{analytic of h}(i), $h(t)$ is analytic as an $L^2(\pa\Om)$-valued function in $t>0$.
By analytic continuation, the condition $h(t)=\wt h(t)$  for $t\in [T_0-\sigma,T_0]$ implies
that $h(t)=\wt h(t)$ in $L^2(\Ga_0)$ for all $t>0$, i.e.,
\[
\rho_0(y)+\rho_1(y)t^\al+\sum_{k\in\BK\cap\BN^*}\rho_k(y)E_{\al,1}(-\la_k t^\al)=\wt\rho_0(y)+\wt\rho_1(y)t^{\wt\al}+\sum_{k\in\wt\BK\cap\BN^*}\wt\rho_k(y)E_{\wt\al,1}(-\wt\lambda_k t^{\wt\al}).
\]
From the decay property of $E_{\al,1}(-\eta)$ (see Lemma \ref{lem:ML-asymp}), we derive $\rho_0(y)+\rho_1(y)t^\al=\wt\rho_0(y)+\wt\rho_1(y)t^{\wt\al}$,
indicating $\rho_0=\wt\rho_0$ and $\rho_1=\wt\rho_1$. Moreover,  we have $\al=\wt\al$ if $1\in\BK$. If $1\not\in\BK$ and $1\not\in\wt\BK$, i.e., $\rho_1=\wt\rho_1=0$, then
\[
\sum_{k\in\BK\cap\BN^*}\rho_k(y)E_{\al,1}(-\la_k t^\al)=\sum_{k\in\wt\BK\cap\BN^*}\wt\rho_k(y)E_{\wt\al,1}(-\wt\lambda_k t^{\wt\al})\quad\mbox{on }\Ga_0\times(0,\infty).
\]
Proposition \ref{analytic of h}(ii) and Laplace transform give
\[
\sum_{k\in\BK\cap\BN^*}\f{\rho_k(y)z^{\al-1}}{z^\al+\la_k}=\sum_{k\in\wt\BK\cap\BN^*}\f{\wt\rho_k(y)z^{\wt\al-1}}{z^{\wt\al}+\wt\lambda_k}.
\]
Assuming that $\al>\wt\al$, dividing both sides by $z^{\wt\al-1}$ and setting $\ze:=z^\al$, we have
\[
\sum_{k\in\BK\cap\BN^*}\f{\rho_k(y)\ze^{1-\frac{\wt\al}{\al}}}{\ze+\la_k}=\sum_{k\in\wt\BK\cap\BN^*}\f{\wt\rho_k(y)}{\ze^{\frac{\wt\al}{\al}}+\wt\la_k}.
\]
Upon noting $\BK\ne\emptyset$, choosing an arbitrary $k_0\in\BK$ and rearranging terms, we derive
\[
\rho_{k_0}(y)\ze^{1-\frac{\wt\al}{\al}}=\left(\sum_{k\in\wt\BK\cap\BN^*}\f{\wt\rho_k(y)}{\ze^{\frac{\wt\al}{\al}}+\wt\la_k} -\sum_{k\in\BK\cap\BN^*\setminus\{k_0\}}\f{\rho_k(y)\ze^{1-\frac{\wt\al}{\al}}}{\ze+\la_k}\right)(\ze+\la_{k_0}).
\]
Letting $\ze\to-\la_{k_0}$ and noting $\al>\wt\al$, the right hand side tends to zero (since all $\wt\la_k$ are positive, and $\arg((-\la_{k_0})^\frac{\wt\al}{\al})=\frac{\wt\al\pi}{\al}\in(0,\pi)$) and hence $\rho_{k_0}\equiv0$ in $L^2(\Gamma_0)$, which contradicts the assumption $k_0\in\BK$. Thus, we deduce $\al\le\wt\al$. The same argument yields $\al\ge\wt\al$, so  $\al=\wt\al$. These discussions thus yield
\begin{align}\label{rho-lambda}
\sum_{k\in\BK\cap\BN^*}\f{\rho_k(y)}{\ze+\la_k}=\sum_{k\in\wt\BK\cap\BN^*}\f{\wt\rho_k(y)}{\ze+\wt\la_k}.
\end{align}
Note that both sides of the identity \eqref{rho-lambda} are $L^2(\Ga_0)$-valued functions in $\ze$. Next we show both converge uniformly in any compact subset in $\BC\setminus(\{-\la_k\}_{k\in\BK\cap\BN^*}\cup\{-\wt\la_k\}_{k\in\wt\BK\cap\BN^*})$ and are analytic in $\BC\setminus(\{-\la_k\}_{k\in\BK\cap\BN^*}\cup\{-\wt\la_k\}_{k\in\wt\BK\cap\BN^*})$. Indeed, since $u_0,f\in L^2(\Om)$, for all $\ze$ in any compact subset of $\BC\setminus(\{-\la_k\}_{k\in\BK\cap\BN^*}\cup\{-\wt\la_k\}_{k\in\wt\BK\cap\BN^*})$, we have
\begin{align*}
&\left\|\sum_{k\in\BK\cap\BN^*}\f{\rho_k}{\ze+\la_k}\right\|_{{\rm Dom}(A)}^2\le c\sum_{\ell\in\BN^*}\la_\ell^2\f{|(u_0,\vp_\ell)|^2+\la_\ell^{-2}|(f,\vp_\ell)|^2}{|\ze+\la_\ell|^2}\\
\le& c\sum_{\ell\in\BN^*}\left(|(u_0,\vp_\ell)|^2+\la_\ell^{-2}|(f,\vp_\ell)|^2\right)<\infty.
\end{align*}
Hence, by the trace theorem, the identity \eqref{rho-lambda} holds for all $\ze\in\BC\setminus(\{-\la_k\}_{k\in\BK\cap\BN^*}\cup\{-\wt\la_k\}_{k\in\wt\BK\cap\BN^*})$.
Assume that $\la_j\not\in\{\wt\la_k\}_{k\in\wt\BK\cap \BN^*}$ for some $j\in\BK\cap\BN^*$. Then we can choose
a small circle $C_j$ centered at $-\la_j$ which does not contain $\{-\wt\la_k\}_{k\in\wt\BK\cap \BN^*}$.
Integrating on $C_j$ and applying the Cauchy theorem give $2\pi\sqrt{-1}\,\rho_j/\la_j=0$, which contradicts the assumption $\rho_j\not\equiv 0$ in $L^2(\Ga_0)$. Hence, $\la_j\in\{\wt\la_k\}_{k\in\wt\BK\cap \BN^*}$ for every $j\in\BK\cap \BN^*$. Likewise, $\wt\la_j\in\{\la_k\}_{k\in\BK\cap\BN^*}$ for every $j\in\wt\BK\cap\BN^*$, and hence $\{\la_k\}_{k\in\BK\cap\BN^*}=\{\wt\la_k\}_{k\in\wt\BK\cap\BN^*} $. From \eqref{rho-lambda}, we obtain
\[
\sum_{k\in\BK\cap\BN^*}\f{\rho_k(y)-\wt\rho_k(y)}{\ze+\la_k}=0,\quad\forall\ze\in\BC\setminus\{-\la_k\}_{k\in\BK\cap\BN^*}.
\]
Varying $j\in\BK\cap\BN^*$ and integrating over $C_j$, we obtain $2\pi\sqrt{-1}\,(\rho_j-\wt\rho_j)/\la_j=0$, which directly implies $\rho_j=\wt\rho_j$ in $L^2(\Ga_0)$. This completes the proof of the theorem.
\end{proof}

\begin{remark}
The condition $\mathbb{K}\neq \emptyset$ holds whenever the following condition is valid $(f,\varphi_1)\neq 0$ or $(u_0,\vp_{\ell,k})-\la_\ell^{-1}(f,\vp_{\ell,k})\neq 0$,  $k=1,\ldots,m_\ell$, $\ell=2,3,\ldots$. Note that the condition
$(f,\varphi_1)\neq 0$ does not rely on the unknown parameter $a$, and can be easily guaranteed.
\end{remark}

The next result gives the uniqueness of recovering the diffusion coefficient $a$
from the lateral boundary observation.

\begin{theorem}\label{thm:unique q}
Let condition \eqref{eq:g-2} be fulfilled, and
let $(a,f,u_0)$, $(\wt a,\wt f,\wt u_0)\in\mathfrak A\times L^2(\Om)\times L^2(\Om)$, and fix
$g$ as \eqref{eq:g}. Let $h$ and $\wt h$ be the corresponding Dirichlet data. Then for any $\si\in(0,T_0]$, the condition $h=\wt h$ on $\Ga_0\times[T_0-\si,T]$ implies  $a=\wt a$.
\end{theorem}

\begin{proof}
In view of the linearity of problem \eqref{eqn:fde}, we can decompose the data $h(t)$ into
\[
h(t)=h_i(t)+h_b(t),\quad t\in(0,T],
\]
with $h_i(t)$ and $h_b(t)$ given by
\begin{align*}
h_i(t) & =\rho_0+\rho_1t^\al+\sum_{k\in\BK\cap\BN^*}\rho_k {E_{\al,1}}(-\la_k t^\al),\\
h_b(t) & =\sum_{\ell=1}^\infty\int_0^t(t-s)^{\al-1}E_{\al,\al}(-\la_\ell(t-s)^\al){ \sum_{k=1}^{m_\ell}\langle g(s),\vp_{\ell,k}\rangle\,\d s\,\vp_{\ell,k}},
\end{align*}
which solve problem \eqref{eqn:fde} with $g\equiv0$ and $f=u_0\equiv0$, respectively. By the
choice of $g$ in \eqref{eq:g}, the interval $[0,T]$ can be divided into two subintervals: $(0,T_0]$ and
$[T_0,T]$. For $t\in(0,T_0)$, $\psi(t)\equiv0$, Theorem \ref{thm:unique al} implies that
$\{(\rho_k,\la_k)\}_{k\in\BK}=\{(\wt\rho_k,\wt\la_k)\}_{k\in\wt
\BK}$ and $\al=\wt\al$, from which we deduce $h_i(t)=\wt h_i(t)$ for
all $ t>0$. For $t\in[T_0,T]$, this and the condition $ h(t)=\wt h(t)$ imply
$h_b(t)=\wt h_b(t)$ in $L^2(\Ga_0)$, and hence
\begin{align*}
&\sum_{\ell=1}^\infty\int_{T_0}^t(t-s)^{\al-1}E_{\al,\al}(-\la_\ell(t-s)^\al)\sum_{k=1}^{m_\ell}\langle g(s),\vp_{\ell,k}\rangle\,\d s\,\vp_{\ell,k}\\
= &\sum_{\ell=1}^\infty\int_{T_0}^t(t-s)^{\al-1}E_{\al,\al}(-\wt\la_\ell(t-s)^\al)\sum_{k=1}^{\wt m_\ell}\langle g(s),\wt\vp_{\ell,k}\rangle\,\d s\,\wt\vp_{\ell,k},\quad t\in[T_0,T].
\end{align*}
By the analyticity in Corollary \ref{analytic of h}, the above identity holds for $t\in[T_0,
\infty)$. Thus applying Laplace transform on both side gives
\begin{align}\label{eq:Laplace_ug}
\sum_{\ell=2}^\infty\f{\sum_{k=1}^{m_\ell}\langle\wh g(z),\vp_{\ell,k}\rangle\vp_{\ell,k}}{z^\al+\la_\ell}=
\sum_{\ell=2}^\infty\f{\sum_{k=1}^{\wt m_\ell}\langle\wh g(z),\wt\vp_{\ell,k}\rangle\wt\vp_{\ell,k}}{z^\al+\wt\la_\ell},\quad\forall\,\Re(z)>0.
\end{align}
Since $\la_1=\wt\la_1=0$ and $\vp_1=\wt\vp_1=|\Om|^{-\frac12}$, the index
in \eqref{eq:Laplace_ug} starts with $\ell=2$. Below we repeat the argument for Theorem
\ref{thm:unique al}. First we show that both sides of \eqref{eq:Laplace_ug} are analytic with
$\ze=z^\al$ in any compact subset of  $\BC\setminus\{-\la_\ell,-\wt\la_\ell\}_{\ell\ge2}$.
Let $U\in \mathrm{Dom}(A^{\frac{1}{4}+\ve})$ be a solution of problem \eqref{eq:U}, for all $\ze$ in a compact subset
of $\BC\setminus\{-\la_\ell,-\wt\la_\ell\}_{\ell\ge2}$, we have
\begin{align*}
 &\left\|\sum_{\ell=2}^\infty\f{\sum_{k=1}^{m_\ell}\langle\wh g(\ze^\frac{1}{\al}),\vp_{\ell,k}\rangle\vp_{\ell,k}}{\ze+\la_\ell}\right\|_{{\rm Dom}(A^{\frac{1}{4}+\ve})}^2
 \le c\sum_{\ell=2}^\infty\la_\ell^{\frac{1}{2}+2\ve}\sum_{k=1}^{m_\ell}\left|\f{\langle\eta,\vp_{\ell,k}\rangle}{\ze+\la_\ell}\right|^2\\
 =&c\sum_{\ell=1}^\infty\la_\ell^{\frac{1}{2}+2\ve}\sum_{k=1}^{m_\ell}\left|\f{\la_\ell(U,\vp_{\ell,k})}{\ze+\la_\ell}\right|^2\le c\|U\|_{{\rm Dom}(A^{\frac{1}{4}+\ve})}^2<\infty.
\end{align*}
Since each term of the series is a $\mathrm{Dom}(A^{\frac14+\ve})$-valued function analytic in
$\ze$ and converges uniformly in $\ze$, by the trace theorem, we obtain that both sides
of  $\eqref{eq:Laplace_ug}$ are $L^2(\pa\Om)$-valued functions analytic in $\ze\in\BC\setminus\{-\la_\ell,-\wt\la_\ell\}_{\ell\ge2}$.
Since $\la_\ell,\wt\la_\ell>0$ for $\ell\ge2$, we may take $\ze\to0$ in \eqref{eq:Laplace_ug} and obtain
\begin{align}\label{eq:eigen}
\sum_{\ell=2}^\infty\f{\sum_{k=1}^{m_\ell}\langle\wh g(0),\vp_{\ell,k}\rangle\vp_{\ell,k}}{\la_\ell}=
\sum_{\ell=2}^\infty\f{\sum_{k=1}^{\wt m_\ell}\langle\wh g(0),\wt\vp_{\ell,k}\rangle\wt\vp_{\ell,k}}{\wt\la_\ell}.
\end{align}
Hence, $ w={\wt w}$ on $\Ga_0$, where $w$ and ${\wt w}$
are the Dirichlet boundary data with $a$ and $\wt a$ in the elliptic problem
\begin{align}\label{eq:elliptic}
\begin{cases}
-\nb\cdot(a\nb w)=0 & \mbox{in }\Om,\\
a\pa_\nu w=\wh g(0) & \mbox{on }\pa\Om
\end{cases}
\end{align}
with the compatibility condition $\int_\Om w\,\d x=0$.
Indeed, the solution $w$ of \eqref{eq:elliptic} can be represented as
\[
w=\sum_{\ell=2}^\infty\sum_{k=1}^{m_\ell}(w,\vp_{\ell,k})\vp_{\ell,k}=\sum_{\ell=2}^\infty\sum_{k=1}^{m_\ell}\la_\ell^{-1}\langle\wh g(0),\vp_{\ell,k}\rangle\vp_{\ell,k},
\]
where the first equality follows from the compatibility condition $\int_\Om w\,\d x=0$ and
the second is due to integration by part. By the choice of $g$ in \eqref{eq:g}, the elliptic
problem \eqref{eq:elliptic} is uniquely solvable. Then from \cite[Theorem 1.1]{FriedmanIsakov:1989},
we deduce that $D=\wt D$ is uniquely determined by the input $\wh g(0)
=\wh{\psi}(0)\eta$. Indeed, Friedman and Isakov \cite{FriedmanIsakov:1989} proved the unique
determination of the convex polygon $D$ for the case $\mu\equiv1$, based on extending the solution
$w$ harmonically across a vertex of $D$ and leading a contradiction. The proof does
not depend on the knowledge of the parameter $\mu$ and hence it is also applicable here. Once
 $D$ is  determined, it suffices to show the uniqueness of the scalar $\mu$. Suppose
$\mu\le\wt\mu$, i.e., $a\le\wt a$ in $D$ and $a\equiv\wt a\equiv1$ outside $D$. Thus $w$ and
$\wt w$ are harmonic functions near $\pa\Om$ with identical Cauchy data on $\Ga_0$, we conclude
$w=\wt w$ near $\pa\Om$. By multiplying both sides of the governing equation in
\eqref{eq:elliptic} with $w$, integrating over the domain $\Omega$ and applying Green's formula, we have
\begin{align*}
 0 & = \int_\Omega -\nabla \cdot(a\nabla w) w \,\d x = \int_\Omega a|\nabla w|^2 \,\d x - \int_{\partial\Omega}w\,\pa_\nu w\,\d S,
\end{align*}
i.e.,
\begin{align*}
  \int_\Omega a|\nabla w|^2 \,\d x = \int_{\partial\Omega}w\,\pa_\nu w\,\d S.
\end{align*}
Now since $w$ and $\widetilde w$ have identical Cauchy data on the boundary $\partial\Omega$, we have
$\int_{\partial\Omega}w\,\pa_\nu w\,\d S = \int_{\partial\Omega}\widetilde w\,\pa_\nu \widetilde w\,\d S$, and consequently
\begin{align*}
  \int_\Omega a|\nabla w|^2 \,\d x = \int_{\Omega}\widetilde a|\nabla \widetilde w|^2\,\d x.
\end{align*}
This identity and the inequality $\widetilde a\geq a$ a.e. in $\Omega$ imply
\begin{align*}
  \int_\Omega a|\nabla w|^2 \,\d x \geq \int_{\Omega} a|\nabla \widetilde w|^2\,\d x,
\end{align*}
which immediately implies
\begin{align*}
  \frac{1}{2}\int_\Omega a|\nabla w|^2\d x - \int_{\partial\Omega}w \widehat g(0) \,\d S  \geq \frac{1}{2}\int_{\Omega} a|\nabla \widetilde w|^2\,\d x - \int_{\partial\Omega} \widetilde w \widehat g(0)\,\d S.
\end{align*}
By the Dirichlet principle \cite{Courant:1950}, $w$ is the minimizer of the energy integral, and hence $w=\wt w$ and $a=\wt a$.
\end{proof}

\begin{remark}\label{rem:poly-ball}
Note that the uniqueness of the inclusion $D$ in \cite{FriedmanIsakov:1989} relies on the assumption $D$ being a convex polygon with $\mathrm{diam}(D)<\mathrm{dist}(D,\pa \Omega)$. Alessandrini \cite{alessandrini1996analicity} removed the diameter assumption for a specialized choice of the boundary data. The works \cite{seo1995uniqueness,kang2001note} proved the unique determination of $D$ when $D$ is a disc or ball.
For general shapes, even for ellipses or ellipsoids, this inverse problem appears still open. Note that in the uniqueness proof, the key is the reduction of the problem to the elliptic case, with a nonzero Neumann boundary condition. In particular, the result will not hold if the temporal component $\psi$ vanishes identically over the interval $[0,T]$, i.e., condition \eqref{eq:g-2} does not hold.
\end{remark}

\begin{remark}\label{rem:general-a}
If the diffusion coefficient $a$ is not piecewise constant, it is also possible to show the unique
recovery if the boundary excitation data $g$ is specially designed. For example,
consider problem \eqref{eqn:fde} with a more general elliptic operator
\begin{align}\label{eq:general coeff}
\cA u(x):=-\nb\cdot(a(x)\nb u(x))+q(x)u(x),\quad x\in\ov\Om.
\end{align}
Here $a\in C^2(\ov\Om)$ and $q\in L^\infty(\Om)$ with {$a>0$ in $\ov\Om$} and $q\ge0$ in $\Om$, and
the Neumann data $g$ is constructed as follows.
First, we choose sub-boundaries $\Ga_1$ and $\Ga_2$ such that
$\Ga_1\cup\Ga_2=\pa\Om$ and $\Ga_1\cap\Ga_2\ne\emptyset$.
Let $\chi\in C^\infty(\pa\Om)$ be a cut-off function with $\mathrm{supp}(\chi)=\Ga_1$ and $\chi\equiv1$ on $\Ga_1'$, with $\Ga_1'\subset\Ga_1$ such that $\Ga_1'\cup\Ga_2=\pa\Om$, $\Ga_1'\cap\Ga_2\ne\emptyset$; see Fig. \ref{Fig:boundary} for an illustration of the geometry in the two-dimensional case. Now we fix $0\le T_0<T_1<T$ and choose a strictly increasing sequence $\{t_k\}_{k=0}^\infty$ such that $t_0=T_0$ and $\lim_{k\to\infty}t_k=T_1$. Consider a sequence $\{p_k\}_{k=1}^\infty\subset\BR_+$ and a sequence $\{\psi_k\}_{k=1}^\infty\subset C^\infty([0,\infty);\ov{\BR_+})$ such that
\[
\psi_k=\begin{cases}
0 & \mbox{on }[0,t_{2k-1}],\\
p_k & \mbox{on }[t_{2k},\infty).
\end{cases}
\]
Then we fix $\{b_k\}_{k=0}^\infty\subset\BR_+$ such that $\sum_{k=1}^\infty b_k\|\psi_k
\|_{W^{2,\infty}(\BR_+)}<\infty$, and define the Neumann data $g$ by
\begin{align}\label{eq:g_series}
g(y,t):=\sum_{k=1}^\infty g_k(y,t)=\chi\sum_{k=1}^\infty b_k\psi_k(t)\eta_k(y),
\end{align}
where the set $\{\eta_k\}_{k=1}^\infty$ is chosen to be dense in $H^{\frac12}(\pa\Om)$ and
$\|\eta_k \|_{H^{\frac12}(\pa\Om)}=1$. Note that the Neumann data $g$ defined in
\eqref{eq:g_series} plays the role of infinity measurements \cite{CanutoKavian:2001,
CanutoKavian:2004}, and hence the unique recovery of the fractional order $\al$ and both
$a$ and $q$ from one boundary measurement. We provide a detailed proof in the appendix for completeness. See
also some related discussions in \cite{KianLiLiuYamamoto:2021,Kian:2022} with different
problem settings.  However, this choice of $g$ is impossible to numerically realize
in practice, due to the need to numerically represent infinitesimally small quantities.
\begin{figure}[hbt!]
	\centering
		\includegraphics[width=0.5\textwidth]{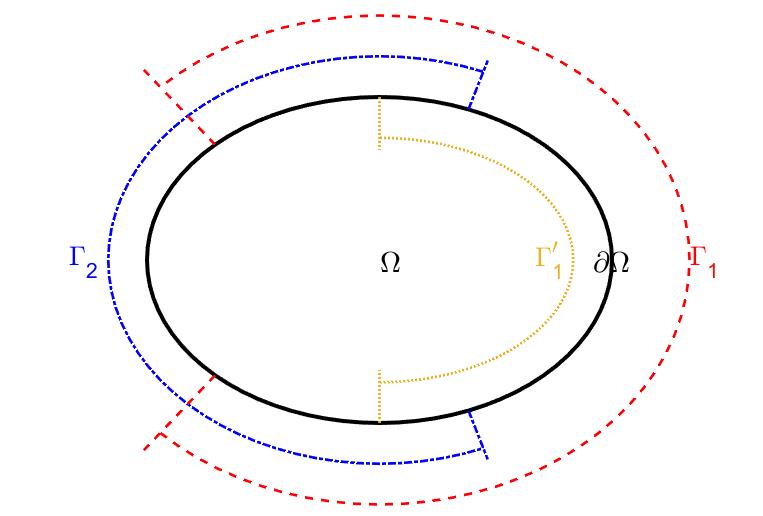}
	\caption{The schematic illustration of the sub-boundaries $\Ga_1'$, $\Ga_1$ and $\Ga_2$ of the boundary $\partial\Omega$.}
	\label{Fig:boundary}
\end{figure}
\end{remark}


\section{Reconstruction algorithm}\label{sec:alg}
In this section, we derive an algorithm for recovering the fractional order $\al$ and the coefficient
$a$, directly inspired by the uniqueness proof. We divide the recovery procedure into three steps:
\begin{enumerate}
\item[(i)] use the asymptotic behavior of the solution of problem \eqref{eqn:fde} near $t=0$ to recover $\al$;
\item[(ii)] use analytic extension to extract the solution of problem \eqref{eqn:fde} with zero $f$ and $u_0$;
\item[(iii)] use the level set method \cite{OsherFedkiw:2001} to recover the shape of the unknown medium $D\subset\Om$.
\end{enumerate}

First, we give an asymptotics of the Dirichlet data $h(t)$ of problem \eqref{eqn:fde}. The result is
direct from the representation and properties of $E_{\al,1}(z)$ near $z=0$
and the trace theorem.

\begin{proposition}\label{prop:asym of h}
If $u_0\in {\rm Dom}(A^{1+\frac{s}{2}})$ and $f\in{\rm Dom}(A^{\frac{s}{2}})$ with $s>1$. Let $ h=u|_{\pa\Om\times(0,T)}$ be the
Dirichlet trace of the solution to problem \eqref{eqn:fde} with $g$ given as \eqref{eq:g},
then the following asymptotic holds:
\[
h(y,t)=u_0(y)+(\cA u_0-f)(y)t^\al+O(t^{2\al})\quad\mbox{as }t\to0^+.
\]
\end{proposition}

In view of Proposition $\ref{prop:asym of h}$, for any fixed $y_0\in\pa\Om$, the asymptotic behavior of $h(y_0,t)$ as $t\to0^+$ allows recovering the order $\al$. This can be achieved by minimizing the following objective in $\al$, $c_0$ and $c_1$:
\begin{align}\label{min al}
J(\al,c_0,c_1)=\|c_0+c_1t^\al-h(y_0,t)\|_{L^2(0,t_0)}^2,
\end{align}
for some small $t_0>0$. Note that it is important to take $t_0$ sufficiently small so
that higher-order terms can indeed be neglected. The idea of using asymptotics for order recovery was
employed in \cite{HatanoYamamoto:2013,JinKian:2021rspa,JinKian:2022siap}.

When recovering the diffusion coefficient $a$, we need to deal with the unknown functions
$u_0$ and $f$. This poses significant computational challenges since standard regularized reconstruction
procedures \cite{EnglHankeNeubauer:1996} require a fully known forward operator. To overcome the challenge, we appeal to
Theorem \ref{thm:unique q}: $u_0$ and $f$ only contribute to $h_i(t) $ which is fully determined
by $\{\la_\ell,\rho_\ell\}_{\ell\in\BK}$. Indeed, by Theorem \ref{thm:unique al}, $\{\la_\ell,
\rho_\ell\}_{\ell\in\BK}$ can be uniquely determined by $h(t)$, $t\in[0,T_0]$.
Hence in theory we can extend $h(t)=h_i(t)$ from $t\in[0,T_0]$ to $t\in[0,T]$, by
means of analytic continuation, to approximate the Dirichlet data of \eqref{eqn:fde}
with $g\equiv0$ and given $u_0$ and $f$. In practice, we look for approximations of the form
\[
h(t)\approx\f{p_0+p_1t+\cdots+p_r t^r}{q_0+q_1t+\cdots+q_r t^r}:=h_r(t),\quad t\in[0,T],
\]
where $r\in\BN$ is the polynomial order. This choice is motivated by the observation that
Mittag-Leffler functions can be well approximated by rational polynomials \cite{AtkinsonOsseiran:2011,Mainardi:2014,DuanZhang:2021}.
The approximation $h_r$ can be constructed efficiently by the AAA algorithm \cite{NakatsukasaTrefethen:2018}.
Now, we can get the Dirichlet data of problem \eqref{eqn:fde} with a given $g$ and $u_0=f\equiv0$,
by defining the reduced data
\[
\ov h(t):=\begin{cases}
0, & t\in[0,T_0],\\
h(t)-h_r(t), & t\in[T_0,T].
\end{cases}
\]

Below we use the reduced data $\ov h$ to recover a piecewise constant $a$.
Parameter identification for the subdiffusion model is commonly carried out by minimizing a
suitable penalized objective. Since $a$ is piecewise constant, it suffices to recover
the interface between different media. The level set method can effectively
capture the interface in an elliptic problem \cite{Santosa:1995,ItoKunischLi:2001,Burger:2001,
ChungChanTai:2005}, which we extend to the time-fractional model \eqref{eqn:fde} below.
Specifically, we consider a slightly more general setting where the inclusion $D\subset\Om$ has
a diffusivity value $a_1$ and the background $\Om\setminus D$ has a diffusivity value $a_2$,
with possibly unknown $a_1$ and $a_2$. That is, the diffusion coefficient $a$ is represented as
\begin{align}\label{eq:level set a}
a(x)=a_1H(\phi(x))+a_2(1-H(\phi(x)))\quad\mbox{in }\Om,
\end{align}
where $H(x)$ and $\phi(x)$ denote the Heaviside function and level set function (a signed distance function):
\[
H(x):=\begin{cases}
1, & x\ge0,\\
0, & x<0,
\end{cases}
\quad\mbox{and}\quad\phi(x):=\begin{cases}
d(x,\pa D), & x\in D,\\
-d(x,\pa D), & x\in\Om\setminus\ov D,
\end{cases}
\]
respectively. Then $\phi $ satisfies $D=\{x\in\Om:\phi(x)>0\}$, $ \Om
\setminus\ov D=\{x\in\Om:\phi(x)<0\}$ and $\pa D=\{x\in\Om:\phi(x)=0\}$.
To find the values $a_1$ and $a_2$ and the interface $\pa D$, we minimize the following functional
\begin{align}\label{min phi}
J(\phi,a_1,a_2)={ \f12\|u(a)-\ov h\|_{L^2(0,T;L^2(\Ga_0))}^2}+\be\int_\Om|\nb a|,
\end{align}
where $u(a)$ is the solution to problem \eqref{eqn:PDE-uib}, and $\be>0$ is the
penalty parameter. The total variation term $\int_\Om|\nb a|$ is to stabilize
the inverse problem, which is defined by 
\[
\int_\Om|\nb a|:=\sup_{ \vp\in(C_0(\ov\Om))^d,|\vp|\le1}\int_\Om a\nb\cdot\vp\,\d x,
\]
where $|\cdot|$ denotes the Euclidean norm. Then we apply the standard gradient descent method
to minimize problem \eqref{min phi}. The next result gives the gradient of $J$.
The notations $J_{T-}^{1-\al}$ and $D_{T-}^\al$ denote the backward
Riemann-Liouville integral and derivative,
defined respectively by \cite[Sections 2.2 and 2.3]{Jin:book2021}
\begin{align*}
J_{T-}^{1-\al}v(t)&:=\f1{\Ga(1-\al)}\int_t^T(s-t)^{-\al}v(s)\,\d s,\\
D_{T-}^\al v(t)&:=-\f1{\Ga(1-\al )}\f\d{\d t}\int_t^T(s-t)^{-\al }v(s)\,\d s.
\end{align*}

\begin{proposition}\label{prop:grad}
The derivative $\f\d{\d a}J$ is formally given by
\[
\f\d{\d a}J(a)=-\int_0^T\nb u\cdot\nb v\,\d t-\be\nb\cdot\left(\f{\nb a}{|\nb a|}\right),
\]
where $ v=v(x,t;a)$ solves the adjoint problem
\begin{align}\label{adjoint}
\begin{cases}
D_{T-}^\al v-\nb\cdot(a\nb v)=0 & \mbox{in }\Om\times[0,T),\\
a\pa_\nu v=(u-\ov h)\chi_{\Ga_0} & \mbox{on }\pa\Om\times[0,T),\\
J_{T-}^{1-\al} v(\cdot,T)=0 & \mbox{in }\Om.
\end{cases}
\end{align}
\end{proposition}

\begin{proof}
We write $J(a)=J_1(a)+J_2(a)$, with $J_1(a)=\f12\| u(a)-\ov h\|_{L^2(0,T;L^2(\Ga_0))}^2$ and $J_2(a)=\be\int_\Om|\nb a|$. For the term $J_1$, the directional derivative along $b$ is
\[
\left.\f\d{\d\ve}\right|_{\ve=0}J_1(a+\ve b)=\int_0^T\!\!\!\int_{\Ga_0}(u(a)-\ov h)u'(a)[b]\,\d S\d t,
\]
where $u'(a)[b]$ is the directional derivative with respect to $a$ in the direction $b$.
Let $\wt a=a+\ve b$ and $\wt u$ solves problem \eqref{eqn:PDE-uib} with the coefficient $\wt a$.
Then $w:=u'(a)[b]=\lim_{\ve\to0}\ve^{-1}(\wt u-u)$. Upon subtracting the equations for $\wt u$
and $u$ and then taking limits, we get
\[
\begin{cases}
\pa_t^\al w-\nb\cdot (a\nb w)=\nb\cdot(b\nb u) & \mbox{in }\Om\times(0,T],\\
a\pa_\nu w=-b\pa_\nu u & \mbox{in }\pa\Om\times(0,T],\\
w(0)=0 & \mbox{in }\Om.
\end{cases}
\]
Multiplying the equation for $w$ with any $\psi\in L^2(0,T;H^1(\Om))$
and integrating over $\Om\times(0,T)$ give
\begin{align}\label{eq:weak u_delta}
\int_0^T\!\!\!\int_\Om(\psi\,\pa_t^\al w+a\nb w\cdot\nb\psi)\,\d x\d t=-\int_0^T\!\!\!\int_\Om b\nb u\cdot\nb\psi\,\d x\d t.
\end{align}
Let $v$ be the solution of problem \eqref{adjoint}. Multiplying the governing equation
for $v$ with a test function $\psi$ and integrating over $\Om\times(0,T)$ give
\begin{align}\label{eq:weak adjoint}
\int_0^T\!\!\!\int_\Om\left(\psi\,D_{T-}^\al v+a\nb v\cdot\nb\psi\right)\d x\d t=\int_0^T\!\!\!\int_{\Ga_0}(u-\ov h)\psi\,\d S\d t.
\end{align}
Note that the following integration by parts formula for fractional derivatives:
\begin{align}\label{eq:integral by parts}
\int_0^T v\,\pa_t^\al w\,\d t=\left[w\,J_{T-}^{1-\al}v\right]_{t=0}^{t=T}+\int_0^T w\,D_{T-}^\al v\,\d t=\int_0^T w\,D_{T-}^\al v\,\d t
\end{align}
(for suitably smooth $v$ and $w$ with $w(0)=0$ and $ J_{T-}^{1-\al}v=0$).
Now by choosing $\psi=v$ in \eqref{eq:weak u_delta}, $\psi=w$ in \eqref{eq:weak adjoint} and applying \eqref{eq:integral by parts}, we obtain
\[
-\int_0^T\!\!\!\int_\Om b\nb u\cdot\nb v\,\d x\d t=\int_0^T\!\!\!\int_{\Ga_0}(u-\ov h)w\,\d S\d t,
\]
implying $\f\d{\d a}J_1(a)=-\int_0^T\nb u\cdot \nb v\d t$. For the term $J_2$, the directional derivative along $b$ is
\begin{align*}
&\left.\f\d{\d\ve}\right|_{\ve=0}\int_\Om|\nb (a+\ve b)|\,\d x =\int_\Om\left.\f\d{\d\ve}\right|_{\ve=0}\left(|\nb (a+\ve b)|^2\right)^{1/2}\d x\\
 =&\int_\Om\left.\left(|\nb (a+\ve b)|^2\right)^{-\f12}\right|_{\ve=0}\nb a\cdot\nb b\,\d x=\int_\Om\f{\nb a}{|\nb a|}\cdot\nb b\,\d x,
\end{align*}
and hence we have $\f\d{\d a}J_2(a)=-\be\nb\cdot(\f{\nb a}{|\nb a|})$.
\end{proof}

By the chain rule, the derivatives of $J$ with respect to $a_1$, $a_2$ and $\phi$ are given by
\begin{align*}
\f{\pa J}{\pa\phi} & =\f{\d J}{\d a}\f{\pa a}{\pa\phi}=\f{\d J}{\d a}(a_1-a_2)\de(\phi),\\
\f{\pa J}{\pa a_1} & =\int_\Om\f{\d J}{\d a}\f{\pa a}{\pa a_1}\,\d x=\int_\Om\f{\d J}{\d a}H(\phi)\,\d x,\\
\f{\pa J}{\pa a_2} & =\int_\Om\f{\d J}{\d a}\f{\pa a}{\pa a_2}\,\d x=\int_\Om\f{\d J}{\d a}(1-H(\phi))\,\d x,
\end{align*}
where $\de$ is the Dirac delta function. Hence the iterative scheme for updating $a_1$, $a_2$ and $\phi$ reads
\[
\phi^{k+1}=\phi^k-\ga^k{ \f{\pa J}{\pa\phi}}(\phi^k,a_1^k,a_2^k)\quad\mbox{and}\quad
a_j^{k+1}=a_j^k-\ga_j^k{ \f{\pa J}{\pa a_j}}(\phi^{k+1},a_1^k,a_2^k),\quad j=1,2.
\]
The step sizes $\ga^k$ and $\ga_j^k$ can be either fixed or obtained by means of line search. The implementation
of the method requires some care. First, we approximate the delta function $\de(x)$ and Heaviside
function $H(x)$ by
\[
 \de_\ve(x)=\f\ve{\pi (x^2+\ve^2)}\quad\mbox{and}\quad
H_\ve(x)=\f1\pi\arctan\left(\f x\ve\right)+\f12,
\]
respectively, with $\ve>0$ of order of the mesh size \cite{ChanTai:2004,ChungChanTai:2005}.
Second, during the iteration, the new iterate of $\phi$ may fail to be a
signed distance function. Although one is only interested in $\mathrm{sign}(\phi)$, it is
undesirable for $|\phi|$ to get too large near the interface. Thus we reset $\phi$ to a
signed distance function whenever $\phi$ changes by more than $10\%$ in the relative $L^2(\Om)$-norm.
The resetting procedure is to find the steady solution of the following equation
\cite{OsherFedkiw:2001,ChungChanTai:2005}:
\[
\pa_t d+\mathrm{sign}(d)(|\nb d|-1)=0,\quad\mbox{with }d(0)=\phi.
\]

\section{Numerical Experiments and Discussions}\label{sec:numer}
Now we present numerical results for reconstructing the fractional order $\al$ and piecewise
constant diffusion coefficient $a$, with unknown $u_0$ and $f$. In all experiments, the domain
$\Om$ is taken to be the unit square $\Om=(0,1)^2$, and the final time $T=1$. We divide
the domain $\Om$ into uniform squares with a length $h=0.02$ and then divide along the diagonals
of each square. We discretize the time interval $[0,T]$ into uniform subintervals with a time
step size $\tau=0.01$. All direct and adjoint problems are solved by standard continuous
piecewise linear Galerkin finite element method in space and backward Euler convolution quadrature in time (see e.g.,\cite{JinLazarovZhou:2019overview} and \cite[Chapters 2 and 3]{JinZhou:2023book}). Below we investigate the following four cases:
\begin{enumerate}
\item[(i)] $D$ is a disc with radius $\frac13$, centered at $(\frac12,\frac12)$,
\item[(ii)] $D$ is a square with length $\frac12$, centered at $(\frac12,\frac12)$,
\item[(iii)] $D$ is a concave polygon, and
\item[(iv)] $D$ is two discs with radius $\frac15$, centered  at $(\frac14,\frac12)$ and $(\frac34,\frac12)$, respectively.
\end{enumerate}
Throughout, the unknown initial condition $u_0$ and source $f$ are fixed as
\[
u_0(x_1,x_2)=x_1^2x_2^2(1-x_1)^2(1-x_2)^2\quad\mbox{and}\quad f(x_1,x_2)=1+x_1+z_2,
\]
respectively. Meanwhile, we fix the exact fractional order $\al^\dag=0.8$ and the diffusion coefficient
$a^\dag=10-9\chi_D$, i.e. $a_1=1$, $a_2=10$. Unless otherwise stated, the Neumann
excitation $g$ is taken as $g(y,t)=\eta(y)\chi_{[0.5,1]}(t)$, where $\eta$ is the cosine
function with a frequency $2\pi$ on each edge for cases (i)--(iii) and is constant $1$ for case (iv).
We set $g$ on $\pa\Om\times[0,T]$, and take the measurement
$h$ on $\pa\Om\times[0,T]$.

\begin{table}[hbt!]\centering
\caption{The recovered order $\al$ based on least-squares fitting.\label{tab:al}}
\begin{threeparttable}
\subfigure[case (i)]{
\begin{tabular}{cccc}
\toprule
$t_0\backslash\al$ &0.3000 &0.5000 &0.8000\\
\midrule
1e-3 & 0.2402 & 0.5289 & 0.8353 \\
1e-4 & 0.2516 & 0.5244 & 0.8795\\
1e-5 & 0.2649 & 0.4994 & 0.8006\\
1e-6 & 0.2712 & 0.4637 & 0.7978\\
1e-7 & 0.2665 & 0.5267 & 0.8019\\
1e-8 & 0.2558 & 0.4913 & 0.7989\\
1e-9 & 0.2744 & 0.4925 & 0.7999 \\
\bottomrule
\end{tabular}}
\subfigure[case (ii)]{
\begin{tabular}{cccc}
\toprule
$t_0\backslash\al$ &0.3000 &0.5000 &0.8000\\
\midrule
1e-3 & 0.2380 & 0.5243 & 0.8350 \\
1e-4 & 0.2479 & 0.5239 & 0.8797\\
1e-5 & 0.2612 & 0.5022 & 0.7803\\
1e-6 & 0.2695 & 0.5182 & 0.7977\\
1e-7 & 0.2662 & 0.5279 & 0.8019\\
1e-8 & 0.2562 & 0.4914 & 0.7989\\
1e-9 & 0.2741 & 0.4925 & 0.7999 \\
\bottomrule
\end{tabular}}\\
\subfigure[case (iii)]{
\begin{tabular}{cccc}
\toprule
$t_0\backslash\al$ &0.3000 &0.5000 &0.8000\\
\midrule
1e-3 & 0.2383 & 0.5214 & 0.8485 \\
1e-4 & 0.2480 & 0.5198 & 0.8821\\
1e-5 & 0.2600 & 0.5098 & 0.8005\\
1e-6 & 0.2667 & 0.5213 & 0.7977\\
1e-7 & 0.2634 & 0.5273 & 0.8019\\
1e-8 & 0.2654 & 0.4913 & 0.7990\\
1e-9 & 0.2718 & 0.4925 & 0.7999 \\
\bottomrule
\end{tabular}}
\subfigure[case (iv)]{
\begin{tabular}{cccc}
\toprule
$t_0\backslash\al$ &0.3000 &0.5000 &0.8000\\
\midrule
1e-3 & 0.2384 & 0.5247 & 0.8436 \\
1e-4 & 0.2486 & 0.5221 & 0.8816\\
1e-5 & 0.2617 & 0.5033 & 0.8005\\
1e-6 & 0.2692 & 0.5178 & 0.7977\\
1e-7 & 0.2650 & 0.5273 & 0.8019\\
1e-8 & 0.2703 & 0.4913 & 0.7989\\
1e-9 & 0.2740 & 0.4925 & 0.7999 \\
\bottomrule
\end{tabular}}
\end{threeparttable}
\end{table}

First, we show the numerical recovery of the fractional order $\al$ for three different values, i.e., 0.3, 0.5 and 0.8. In view of Proposition \ref{prop:asym of h}, it suffices to fix one point $y_0\in\pa\Om$ (which is fixed at the origin $y_0=(0,0)$ below) and to minimize problem \eqref{min al}, for which we use the L-BFGS-B with constraint $\al\in[0,1]$ \cite{ByrdNocedalZhu:1995}. The recovered orders are presented in Table \ref{tab:al}. Note that the least-squares functional has many local minima. Hence, the algorithm requires a good initial guess to get a correct value for $\al$. It is observed that the reconstruction is more accurate when $t_0\to0^+$, since the high order terms are then indeed negligible. Also, for a fixed interval $(0,t_0)$, due to the asymptotic behavior, we have slightly better approximations when the true order $\al$ is large. However, this does not influence much the reconstruction results for cases (i)--(iv), since the coefficient $a$ is constant near origin.

Now we apply analytic continuation to extend the observed data $h$ by a rational function $h_r$ from the interval $[0,0.5]$ to $[0,1]$, using
the AAA algorithm \cite{NakatsukasaTrefethen:2018} with degree $r=4$. This step is essential for dealing with missing data $u_0$ and $f$:  subtracting $h_r$ from $h$ yields the reduced data $\ov h$ for a given $g$ and $u_0=f\equiv0$, which is then used in recovering $a$. Fig. \ref{Fig:ana_error} shows the $L^2(\pa\Om)$ error between $h_r$ and the exact data $h_0$ which is obtained by solving \eqref{eqn:fde} with given $g$ and vanishing $u_0$ and $f$. Note that higher order rational approximations can reduce the error over the interval $[0,0.5]$, but it tends to lead to larger errors in the interval $[0.5,1]$. The approach is numerically sensitive to the presence of data noise, reflecting the well-known severe ill-posed nature of analytic continuation.

\begin{figure}[hbt!]
\begin{tabular}{cccc}
\includegraphics[width=0.23\textwidth]{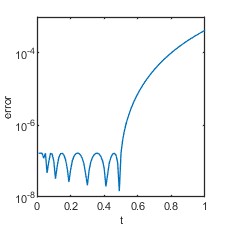}
&\includegraphics[width=0.23\textwidth]{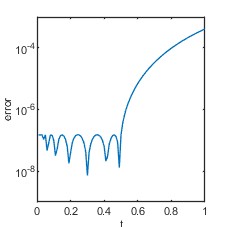}
&\includegraphics[width=0.23\textwidth]{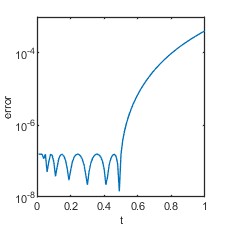}
&\includegraphics[width=0.23\textwidth]{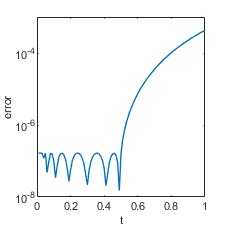}
\end{tabular}
\caption{The $L^2(\pa\Om)$-error between the analytic continuation $h_r$ and true data $h_0$ for cases (i)--(iv).}
\label{Fig:ana_error}
\end{figure}

Finally, we present recovery results for the piecewise constant coefficient $a$, or equivalently, the shape $D$. The exact value is $1$ inside the inclusion $D$ and $10$ outside, unless otherwise stated. We use the standard gradient descent method to minimize problem \eqref{min phi}. Unless otherwise stated, we fix the step sizes $\ga^k\equiv1$, $\ga_1^k\equiv0$, $\ga_2^k\equiv0$, i.e., fixing the values inside and outside the inclusion $D$. The regularization parameter $\be$ is chosen to be $10^{-8}$, and the coefficients $a_1$ and $a_2$ are set to $a_1=0.9$ and $a_2=10$. The results are summarized in Figs. \ref{Fig:disc}-\ref{Fig:noise}, where dashed lines denote the recovered interfaces.

\begin{figure}[hbt!]\centering
\begin{tabular}{ccc}
\includegraphics[width=0.3\textwidth]{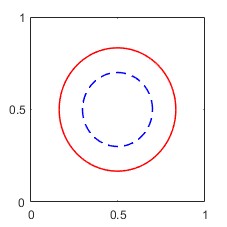}
&\includegraphics[width=0.3\textwidth]{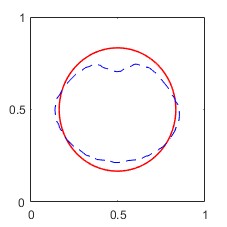}
&\includegraphics[width=0.3\textwidth]{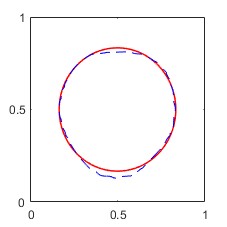}\\
\includegraphics[width=0.3\textwidth]{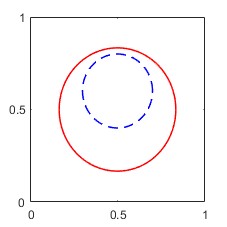}
&\includegraphics[width=0.3\textwidth]{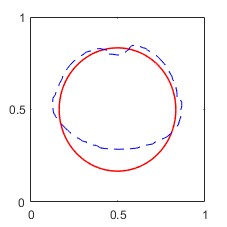}
&\includegraphics[width=0.3\textwidth]{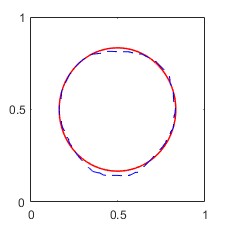}
\end{tabular}
\caption{The reconstructions of the interface for case (i) at iteration $0$, $100$ and $10000$ from left to right, with two different initial guesses.}
\label{Fig:disc}
\end{figure}

Fig. \ref{Fig:disc} shows the result for case (i), when the initial guesses are a small circle
but with two different centers. In either case, the algorithm can successfully reconstruct the exact circle
after $10000$ iterations.  For case (ii), the exact interface is a square, again with the initial guess being small circles inside
the square, cf. Fig. \ref{Fig:square}. The algorithm accurately recovers the four edges of the square. However, due to the non-smoothness,
the corners are much more challenging to reconstruct and hence less accurately resolved. These
results indicate that the method does converge with a reasonable initial guess, but it may
take many iterations to yield satisfactory reconstructions. Fig. \ref{Fig:concave}
shows the results for case (iii) for which the exact interface is a concave polygon,
 which is much more challenging to resolve. Nonetheless, the
algorithm can still recover the overall shape of the interface. The reconstruction
around the concave part has lower accuracy. To the best of our knowledge,
 the unique determination of a concave polygonal inclusion (in an elliptic equation) is still open. Fig.
 \ref{Fig:two_disc} shows the results for case
(iv) which contains two discs as the exact interface. The initial guess is two small
discs near the boundary $\partial\Omega$. Note that in this case, we choose the boundary
data $\eta\equiv1$ in order to strengthen the effect of inhomogeneity. The final
reconstruction is very satisfactory.

\begin{figure}[hbt!]\centering
\begin{tabular}{ccc}
\includegraphics[width=0.3\textwidth]{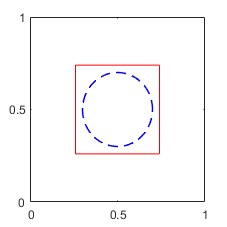}
&\includegraphics[width=0.3\textwidth]{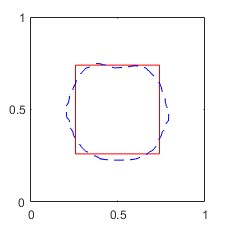}
&\includegraphics[width=0.3\textwidth]{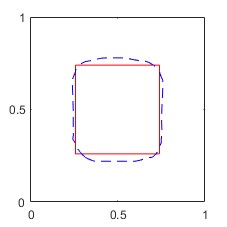}\\
\includegraphics[width=0.3\textwidth]{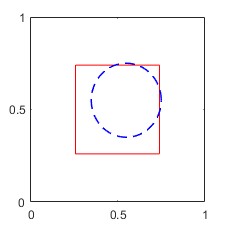}
&\includegraphics[width=0.3\textwidth]{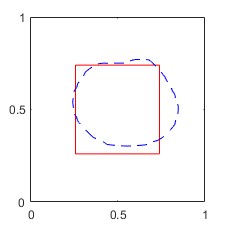}
&\includegraphics[width=0.3\textwidth]{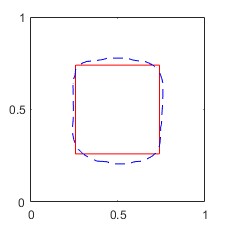}
\end{tabular}
\caption{The reconstructions of the interface for case (ii) with different initial guesses at iteration $0$, $100$ and $8000$ from left to right.}
\label{Fig:square}
\end{figure}

\begin{figure}[hbt!]\centering
\begin{tabular}{ccc}
\includegraphics[width=0.3\textwidth]{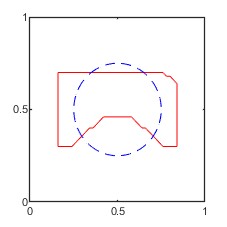}
&\includegraphics[width=0.3\textwidth]{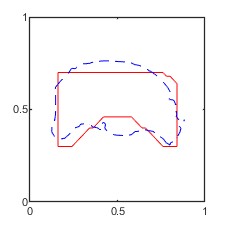}
&\includegraphics[width=0.3\textwidth]{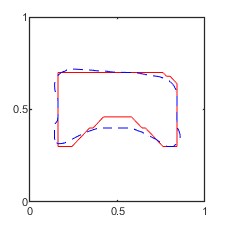}
\end{tabular}
\caption{The reconstructions of the interface for case (iii) at iteration $0$, $100$ and $8000$ from left to right.}
\label{Fig:concave}
\end{figure}

\begin{figure}[hbt!]\centering
\begin{tabular}{ccc}
\includegraphics[width=0.3\textwidth]{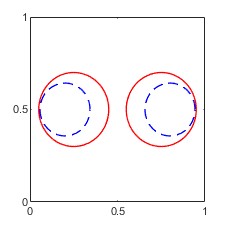}
&\includegraphics[width=0.3\textwidth]{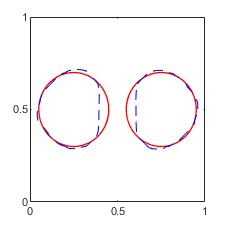}
&\includegraphics[width=0.3\textwidth]{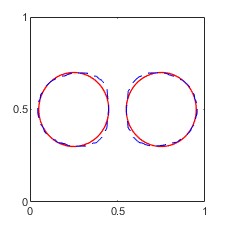}
\end{tabular}
\caption{The reconstructions of the interface for case (iii) at iteration $0$, $1000$ and $15000$ from left to right.}
\label{Fig:two_disc}
\end{figure}

Fig. \ref{Fig:top} shows a variant of case (ii), with the initial interface being
two disjoint discs. It is observed that the two discs first merge into one concave
contour, and then it evolves slowly to resolve the square. This shows one distinct feature
of the level set method, i.e., it allows topological changes. Due to the complex evolution, the
algorithm takes many more iterations to reach convergence (i.e., $30000$ iterations
versus $8000$ iterations in case (ii)).

\begin{figure}[hbt!]\centering
\begin{tabular}{ccc}
\includegraphics[width=0.3\textwidth]{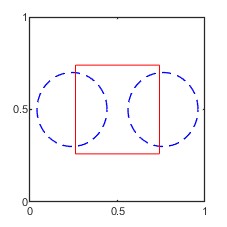}
&\includegraphics[width=0.3\textwidth]{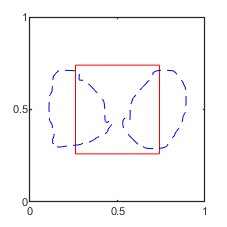}
&\includegraphics[width=0.3\textwidth]{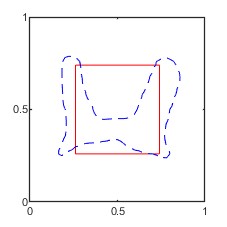}\\
\includegraphics[width=0.3\textwidth]{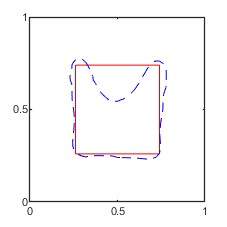}
&\includegraphics[width=0.3\textwidth]{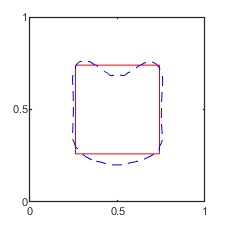}
&\includegraphics[width=0.3\textwidth]{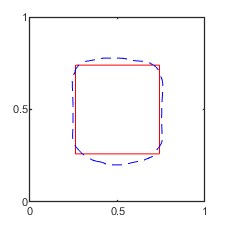}
\end{tabular}

\caption{The reconstruction of the interface for case (ii) with a different initial guess, at different iterations $0$, $100$, $1000$, $10000$, $20000$ and $30000$ (from left to right).}
\label{Fig:top}
\end{figure}

Fig. \ref{Fig:interface_a1} shows a case which aims at simultaneously recovering the interface and the
diffusivity value inside the inclusion, for which the exact interface is a square and the exact values
of $a_1$ and $a_2$ are $1$ and $10$, respectively. In the experiment, we take two different initial
guesses. The initial value of $a_1$ for both cases is $a_1=1.2$, and we take the step sizes $\ga^k\equiv1$,
$\ga_1^k\equiv10$ and $\ga_2^k\equiv0$. The recovered value $a_1$ is $0.92$ for the first row and $0.89$
for the second row. It is observed that for both cases, one can roughly recover the interface. These
experiments clearly indicate that the level set method can accurately recover the
interface $D$. However, it generally takes many iterations to obtain satisfactory results. This is attributed
partly to topological changes and the presence of nonsmooth points, and partly to the direct gradient
flow formulation. Indeed, one observes from Proposition \ref{prop:grad} that the gradient field for
updating the level set function is actually not very smooth, which hinders the rapid evolution of the interface. Hence, there
is an imperative need to accelerate the method, especially via suitable preconditioning and post-processing
\cite{ItoKunischLi:2001}.

\begin{figure}[hbt!]\centering
\begin{tabular}{cc}
\includegraphics[width=0.3\textwidth]{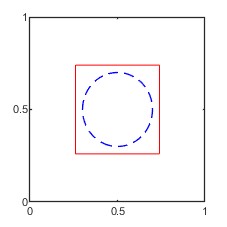}&
\includegraphics[width=0.3\textwidth]{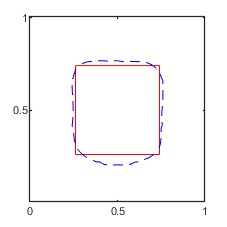}\\
\includegraphics[width=0.3\textwidth]{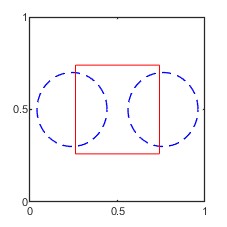} &
\includegraphics[width=0.3\textwidth]{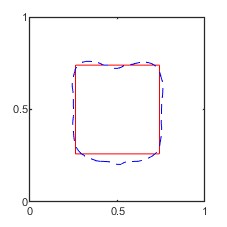}
\end{tabular}
\caption{The reconstructions for case (ii) with a non-fixed diffusivity value $a_1$. Left column: initial guess. Right column: recovered interface.}
\label{Fig:interface_a1}
\end{figure}

Last, Fig. \ref{Fig:noise} shows reconstruction results with noisy data. Due to the instability of
analytic continuation for noisy data, we use boundary data corresponding to zero $u_0$, $f$ as our
measurement and only focus on reconstructing $a$. That is, we denote $h^*$
the solution of problem \eqref{eqn:fde} with  $u_0\equiv0$ and $f\equiv0$ which plays the role of
$\ov h$. The noisy measurement $h^\de$ is generated by
\[
h^\de(y,t)=h^*(y,t)+\ve\| h^*\|_{L^\infty(\pa\Om\times[0,1])}\xi(y,t),
\]
where $\ve>0$ denotes the relative noise level, and $\xi$ follows the standard Gaussian distribution.
We take the exact interface as a concave polygon and the initial guess is a circle; see the
left panel in Fig. \ref{Fig:concave}. We consider two different noise levels and three different
input boundary data. The first and second rows in Fig. \ref{Fig:noise} are for $1\%$ and $5\%$ noise,
obtained with a regularization parameter $\be=1\times10^{-7}$ and $\be=5\times10^{-7}$,
respectively. We consider three input Neumann data $g_1$, $g_2$ and $g_3$: $g_1=g$ (i.e.,
identical as before), and $g_2$ and $g_3$ are given by
\begin{align*}
g_2(x,t) & =\eta_1(x)\chi_{[0.25,1]}(t)+\eta_2(x)\chi_{[0.5,1]}(t)+\eta_3(x)\chi_{[0.75,1]}(t),\\
g_3(x,t) & =\eta_1(x)\chi_{[1/6,1]}(t)+\eta_2(x)\chi_{[2/6,1]}(t)+\eta_3(x)\chi_{[3/6,1]}(t)+\eta_4(x)\chi_{[4/6,1]}(t)+\eta_5(x)\chi_{[5/6,1]}(t),
\end{align*}
where $\eta_n$ ($n=1,\dots,5$) is a cosine function with frequency $2n\pi$ on each edge.
The inputs $g_2$ and $g_3$ contain higher frequency information and are designed to examine
the influence of boundary excitation on the reconstruction.
Fig. \ref{Fig:noise} shows that with the knowledge of $h^*$, the method for
recovering the interface is largely stable with respect to the presence of data noise.
With more frequencies in the input excitation, the reconstruction
results would improve slightly. This agrees with the observation that
the concave shape contains more high-frequency information.

\begin{figure}[hbt!]\centering
\begin{tabular}{ccc}
\includegraphics[width=0.3\textwidth]{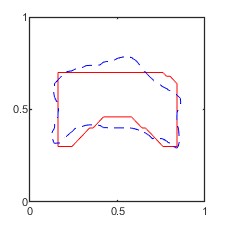}
&\includegraphics[width=0.3\textwidth]{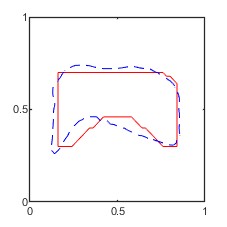}
&\includegraphics[width=0.3\textwidth]{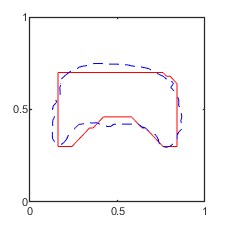}\\
\includegraphics[width=0.3\textwidth]{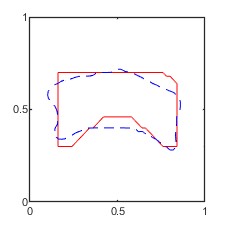}
&\includegraphics[width=0.3\textwidth]{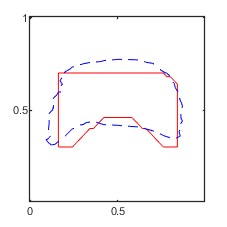}
&\includegraphics[width=0.3\textwidth]{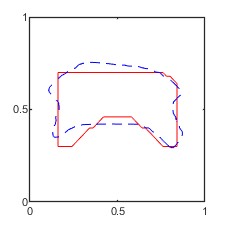}
\end{tabular}
\caption{The reconstruction for case (iii) with noisy data and different
boundary excitations $g_1$, $g_2$ and $g_3$ (from left to right). The
top and bottom rows are for noise levels $1\%$ and $5\%$.}
\label{Fig:noise}
\end{figure}

\section{Concluding remarks}
In this work have studied a challenging inverse problem of recovering multiple coefficients
from one single boundary measurement, in a partially unknown medium, due to the formal under-determined
nature of the problem. We have presented two uniqueness results, i.e., recovering
the order and the piecewise constant diffusion coefficient from a fairly general Neumann input data and
 recovering the order and two distributed parameters from a fairly specialized
Neumann input data (in the appendix). For the former, we have also developed a practical reconstruction
algorithm based on asymptotic expansion, analytic continuation and level set method, which is inspired
by the uniqueness proof, and have
presented extensive numerical experiments to showcase the feasibility of the approach.

There remain many important issues to be resolved. Numerically, the overall algorithmic pipeline
works well for exact data. However, analytic continuation with rational functions is
sensitive with respect to the presence of data noise. Thus it is of much interest to develop
one-shot reconstruction algorithms. The main challenge lies in unknown medium properties,
i.e., missing data, which precludes a direct application of many standard regularization
techniques. It is of much interest to develop alternative approaches for problems with missing
data. The level set method does give excellent reconstructions, but it may take many iterations
to reach convergence. The acceleration of the method, e.g., via preconditioning, is imperative.
Theoretically the specialized Neumann input is very powerful. However, the numerical realization
is very challenging. It would also be interesting to develop alternative numerically feasible
yet more informative excitations for recovering more general coefficients than polygonal inclusions.

\appendix
\section{Recovery of two general coefficients}\label{sec:uniqueness-general}
In this appendix, we discuss the unique recovery of general coefficients mentioned in Remark
\ref{rem:general-a}. In this setting, we have $g\in C^2(\ov\BR_+;H^\frac{1}{2}(\pa\Om))$
with support in $\Ga_1\times\ov\BR_+$. Moreover, $g$ is piecewise constant in time $t$ and $g\equiv0$
when $t\le T_0$, $g$ is constant when $t\ge T_1$. The proof relies on the representation of
the data $h$, similar to Corollary \ref{analytic of h} and hence we omit the proof. Note
that $g$ is a space-time dependent series. We may write $h=h_i+\sum_{k=1}^\infty h_{b,k}$ to
distinguish the contributions from $u_0$ and $f$, and $g$ (with $g_k(t):=\chi b_k\psi_k(t)\eta_k$)
\begin{align*}
h_i(t) & :={\rho_0+\rho_1t^\al+\sum_{n=2}^\infty}\rho_n E_{\al,1}(-\la_n t^\al),\\
h_{b,k}(t) & :=\sum_{n=1}^\infty\sum_{j=1}^{m_n}\int_0^t(t-s)^{\al-1}E_{\al,\al}(-\la_n(t-s)^\al)\langle g_k(s),\vp_{n,j}\rangle\,\d s\,\vp_{n,j}.
\end{align*}

\begin{proposition}\label{prop:sol-rep-gen}
For $u_0,f\in L^2(\Om)$ and $\eta_k\in H^\frac{1}{2}(\pa\Om)$,
the data $h=u|_{\pa\Om\times(0,T)}$ to problem \eqref{eqn:fde} can be represented by
\begin{align*}
h(t)=& \rho_0+\rho_1t^\al+\sum_{n=2}^\infty\rho_n E_{\al,1}(-\la_n t^\al)\\
 &+\sum_{n=1}^\infty\sum_{j=1}^{m_n}\int_0^t(t-s)^{\al-1}E_{\al,\al}(-\la_n(t-s)^\al)\langle g(s),\vp_{n,j}\rangle\,\d s\,\vp_{n,j},
\end{align*}
with $\rho_n$ defined in \eqref{eqn:varrho}. Moreover, the following statements hold.
\begin{enumerate}
\item[(i)] $h_i\in C^\om(0,\infty;L^2(\pa\Om))$ and $h_{b,k}\in C^\om(t_{2k}+\ve,\infty;L^2(\pa\Om))$  with an arbitrarily fixed $\ve>0$.
\item[(ii)] The Laplace transforms  $\wh h_i(z)$ and $\wh h_{b,k}(z)$ of $h_i$ and $h_{b,k}$ in $t$ exist and are given by
\begin{align*}
\wh h_i(z)&= \rho_0z^{-1}+\Ga(1+\al)\rho_1z^{-1-\al}+\sum_{n=2}^\infty\f{\rho_kz^{\al-1}}{z^\al+\la_k},\\
\wh h_{b,k}(z)&=\sum_{n=1}^\infty\sum_{j=1}^{m_n}\f{\langle\wh g_k(z),\vp_{n,j}\rangle\vp_{n,j}}{z^\al+\la_n}.
\end{align*}
\end{enumerate}
\end{proposition}

Now we can state the main result of this part. First, we uniquely determine the fractional order $\al$ using
the data near $t=T_0$, and then use the special boundary excitation $g$ to determine the coefficients $a$ and $q$.
The proof of part (i) is identical with that for Theorem \ref{thm:unique al}, and hence omitted. The unique
determination of $a$ and $q$ is proved below.

\begin{theorem}\label{thm:general}
Let $\al,\wt\al\in(0,1)$, $(a,q,f,u_0),(\wt a,\wt q,\wt f,\wt u_0)\in C^2(\ov\Om)
\times L^\infty(\Om)\times L^2(\Om)\times L^2(\Om)$ and fix $g$ as \eqref{eq:g_series}.
Let $h$ and $\wt h$ be the corresponding Dirichlet data, and let $\si\in(0,T_0]$ be fixed.
\begin{enumerate}
\item[(i)] The condition $h=\wt h$ on $\Ga_2\times[T_0-\si,T_0]$ implies $\al=\wt\al$, $\rho_0=\wt\rho_0$ and $\{(\rho_k,\la_k)\}_{k\in\BK}=\{(\wt\rho_k,\wt\la_k)\}_{k\in\wt\BK}$, if $\BK,\wt\BK\ne\emptyset$.
\item[(ii)] If either of following conditions is satisfied: {\rm(a)} $q=\wt q$ and $a-\wt a=|\nb a-\nb\wt a|=0$ on the boundary $\pa\Om$ or {\rm(b)} $a=\wt a$, then the condition $h=\wt h$ on $\Ga_0\times[T_0-\si,T]$ implies $(a,q)=(\wt a,\wt q)$.
\end{enumerate}
\end{theorem}

In the proof of Theorem \ref{thm:general}, we need the following two lemmas.

\begin{lemma}\label{lem:h_k}
The identity $h=\wt h$ on $\Ga_2\times[T_0-\si,T_1]$ implies
\begin{align}\label{goal1}
h_k=\wt h_k\quad \mbox{on }\Ga_2\times[T_0-\si,\infty),\quad\forall k\in\BN,
\end{align}
with $h_k=h_i+h_{b,k}$ which solves problem \eqref{eqn:fde} with $g=g_k$.
\end{lemma}

\begin{proof}
We prove the assertion by induction. When $k=1$, by the definition of $\psi_k(t)$, we have
$\psi_k=0$  in $(0,t_3)$ for all $k\ge2$. Then by Proposition \ref{prop:sol-rep-gen},
the condition $h=\wt h$ on $\Ga_2\times[T_0-\si,T_1]$ implies $h_1=\wt h_1$ on $\Ga_2\times[T_0-\si,t_3)$,
since $[T_0-\si,t_3)\subset[T_0-\si,T_1]$. By Proposition
\ref{prop:sol-rep-gen}(i), $h_1$ and $\wt h_1$ are $L^2(\pa\Om)$-valued
functions analytic in  $t\in(t_2+\ve,\infty)$, and hence  $h_1=\wt h_1$ for all
$t\in[T_0-\si,\infty)$. This shows the case for $k=1$.
Now assume that for some $\ell\in\BN$, the assertion \eqref{goal1} holds for
all $k=1,\dots,\ell $. Since $\psi_k=0$ in $(0,t_{2\ell+3})$, for $k\ge\ell+2$,
we deduce $\sum_{k=1}^{\ell+1}h_k=h$ in $\Ga_2\times(0,t_{2\ell+3})$.
Similarly, we have
\begin{align*}
\sum_{k=1}^{\ell+1}h_k=\sum_{k=1}^{\ell+1}\wt h_k\quad\mbox{on }\Ga_2\times(0,t_{2\ell+3}).
\end{align*}
From the induction hypothesis, we deduce $h_{\ell+1}=\wt h_{\ell+1}$ on
$\Ga_2\times[T_0-\si,t_{2\ell+3})$. Use analytic continuation again, we
obtain $h_{\ell+1}=\wt h_{\ell+1}$ on $\Ga_2\times[T_0-\si,\infty)$.
Thus, the assertion \eqref{goal1} holds for all $k\in\BN$.
\end{proof}

\begin{lemma}\label{lem:linearly_independent}
Given a nonempty open subset $\Ga$ of $\pa\Om$, for any fixed $n\in\BN^*$, the
eigenfunctions $\{\vp_{n,\ell}\}_{\ell=1}^{m_n}$ corresponding to $\la_n$ are
linearly independent on $L^2(\Ga)$.
\end{lemma}
\begin{proof}
Suppose that on the contrary: there are $\{c_j\}_{j=1}^{m_n}\subset\mathbb{R}$ such that
$\sum_{j=1}^{m_n}c_j\vp_{n,j}=0$ on $\Ga$.
Let $\vp=\sum_{j=1}^{m_n}c_j\vp_{n,j}$. Then $\vp$ satisfies
$\cA\vp=\la_n\vp ~\mbox{in }\Om$, $\pa_\nu\vp=0~ \mbox{on }\pa\Om$,
and $\vp=0 ~ \mbox{on }\Ga$.
Then the regularity on $a$ and $q$ and  unique continuation principle \cite[Theorem 3.3.1]{isakov2017inverse}
imply $\vp\equiv0$ in $\Om$. Since $\vp_{n,j}$ are linearly independent in
$L^2(\Om)$, we obtain $c_j=0$, $j=1,\dots,m_n$, i.e. the desired linear independence.
\end{proof}

Now we can state the proof of Theorem \ref{thm:general}(ii).

\begin{proof}[Proof of Theorem \ref{thm:general}(ii)]
By Lemma \ref{lem:h_k}, we have $h_k=\wt h_k$ on $\Ga_2\times(T_0-\si,\infty)$
for any $k\in\BN$. Note that $h_k=h_i+h_{b,k}$ solves problem \eqref{eqn:fde}
with $g$ replaced by $g_k$. We have the following representations
\begin{align*}
h_i(t) & =\rho_0+\rho_1t^\al+\sum_{n\in\BK\cap\BN^*}\rho_n E_{\al,1}(-\la_n t^\al),\\
h_{b,k}(t) & =\sum_{n=1}^\infty\sum_{j=1}^{m_n}\int_0^t(t-s)^{\al-1}E_{\al,\al}(-\la_n(t-s)^\al)\langle g_k(s),\vp_{n,j}\rangle\,\d s\,\vp_{n,j}.
\end{align*}
By the choice of $g$, the interval $[0,T]$ can be divided into
$[0,T_0]$ and $[T_0,T]$. For $t\in(0,T_0)$, $g_k(t)\equiv0$, Theorem \ref{thm:general}(i)
implies that $\{(\rho_\ell,\la_\ell)\}_{\ell\in\BK}=\{(\wt\rho_\ell,\wt\la_\ell)\}_{\ell\in\wt\BK}$
and $\al=\wt\al$, and hence $h_i(t)=\wt h_i(t)$ for all $t>0$. For $t\in[T_0,T]$,
this and the condition $ h_k(t)=\wt h_k(t)$ lead to $h_{b,k}(t)=\wt h_{b,k}(t)$ in $L^2(\Ga_2)$. Thus,
\begin{align*}
& \sum_{n=1}^\infty\sum_{j=1}^{m_n}\int_{t_1}^t(t-s)^{\al-1}E_{\al,\al}(-\la_n(t-s)^\al)\langle g_k(s),\vp_{n,j}\rangle_{L^2(\Ga_1)}\,\d s\,\vp_{n,j}\\
= & \sum_{n=1}^\infty\sum_{j=1}^{\wt m_n}\int_{t_1}^t(t-s)^{\al-1}E_{\al,\al}(-\wt\la_n(t-s)^\al)\langle g_k(s),\wt\vp_{n,j}\rangle_{L^2(\Ga_1)}\,\d s\,\wt\vp_{n,j},\quad t\in[t_1,\infty).
\end{align*}
By Proposition \ref{prop:sol-rep-gen}(ii), applying Laplace transform on both sides yields
\begin{align}\label{eq:Laplace_ug general}
\sum_{n=1}^\infty\sum_{j=1}^{m_n}\f{\langle\wh g_k(z),\vp_{n,j}\rangle\vp_{n,j}}{z^\al+\la_n}=\sum_{n=1}^\infty\sum_{j=1}^{\wt m_n}\f{\langle\wh g_k(z),\wt\vp_{n,j}\rangle\wt\vp_{n,j}}{z^\al+\wt\la_n},\quad\forall\,\Re(z)>0.
\end{align}
Next, we repeat the argument of Theorems \ref{thm:unique al} and \ref{thm:unique q} to deduce
$\la_n=\wt\la_n$, $\forall n\in\BN$. 
To this end, let $U_k\in \mathrm{Dom}(A^{\frac{1}{4}+\ve})$ be the solution of the elliptic
equation with a Neumann boundary data $\chi b_k\eta_k$, for all $\ze$ in any compact subset of
$\BC\setminus\{-\la_n,-\wt\la_n\}_{n\in\BN}$, we have
\begin{align*}
&\left\|\sum_{n=1}^\infty\sum_{j=1}^{m_n}\f{\langle\wh g_k(\eta^\frac{1}{\al}),\vp_{n,j}\rangle\vp_{n,j}}{\ze+\la_n}\right\|_{{\rm Dom}(A^{\frac{1}{4}+\ve})}^2  \le c\sum_{n=1}^\infty\la_n^{\frac{1}{2}+2\ve}\sum_{j=1}^{m_n}\left|\f{\langle\chi b_k\eta_k,\vp_{n,j}\rangle}{\ze+\la_n}\right|^2\\
=&c\sum_{n=1}^\infty\la_n^{\frac{1}{2}+2\ve}\sum_{j=1}^{m_n}\left|\f{\la_n\left(U_k,\vp_{n,j}\right)}{\ze+\la_n}\right|^2\le c\|U_k\|_{{\rm Dom}(A^{\frac14+\ve})}^2<\infty.
\end{align*}
Since each term of the series is a $\mathrm{Dom}(A^{\frac{1}{4}+\ve})$-valued function analytic in $\ze$ and the series converges uniformly for $\ze$ in a compact subset set of $\BC\setminus\{-\la_n,-\wt\la_n\}_{n\in\BN}$, by the trace theorem, we deduce that both sides of \eqref{eq:Laplace_ug general} are $L^2(\pa\Om)$-valued functions analytic in $\ze\in\BC\setminus\{-\la_n,-\wt\la_n\}_{n\in\BN}$. Assuming $\la_j\notin\{\wt\la_n\}_{n\in\BN}$, by choosing a small circle centered at $-\la_j$ and then using Cauchy integral formula, we obtain
\begin{align}\label{eq:eigenvalue_general}
\f{2\pi\sqrt{-1}}{\la_j}\sum_{j=1}^{m_n}\langle\wh g_k,\vp_{n,j}\rangle\vp_{n,j}(y)=0,\quad\forall k\in\BN.
\end{align}
This and Lemma \ref{lem:linearly_independent} (with $\Ga=\Ga_2$) imply $\langle\wh g_k,\vp_{n,j}\rangle=0$, $\forall k\in\BN$, $j=1,\dots,m_n$. Since $\wh g_k=\chi b_k\wh\psi\eta_k$, by the density of $\eta_k$ in $H^{\frac{1}{2}}(\Ga_1)$, we have $\vp_{n,j}=0$ a.e. on  $\Ga_1$, $j=1,\dots,m_n$. Since $\pa_\nu\vp_{n,j}=0$, unique continuation principle \cite[Theorem 3.3.1]{isakov2017inverse} implies $\vp_{n,j}\equiv0$ in $\Om$, which is a contradiction. Hence, $\la_j\in\{\wt\la_n\}_{n\in\BN}$ for every $j\in\BN$. Likewise, we can prove $\wt\la_j\in\{\la_n\}_{n\in\BN}$ for every $j\in\BN$, and hence $\la_n=\wt\la_n$, $\forall n\in\BN^*$. It follows directly from \eqref{eq:Laplace_ug general} that
\[
\sum_{n=1}^\infty\f1{\eta+\la_n}\Big(\sum_{j=1}^{m_n}\langle\wh g_k(z),\vp_{n,j}\rangle\vp_{n,j}(y)-\sum_{j=1}^{\wt m_n}\langle\wh g_k(z),\wt\vp_{n,j}\rangle\wt\vp_{n,j}(y)\Big)=0,\quad\mbox{a.e. }y\in\Ga_2.
\]
Using Cauchy integral theorem again, we have
\[
\sum_{j=1}^{m_n}\langle\wh g_k(z),\vp_{n,j}\rangle\vp_{n,j}(y)=\sum_{j=1}^{\wt m_n} \langle\wh g_k(z),\wt\vp_{n,j}\rangle\wt\vp_{n,j}(y),\quad\mbox{a.e. }y\in\Ga_2,\ \forall k,n\in\BN.
\]
By the construction of $g_k$, it is equivalent to
\begin{align*}
&b_k\psi_k(z)\int_{\pa\Om}\chi(y')\eta_k(y')\Te_n(y',y)\,\d y'\\
=&b_k\psi_k(z)\int_{\pa\Om}\chi(y')\eta_k(y')\wt\Te_n(y',y)\,\d y',\quad\forall n,k\in\BN,\ \Re(z)>0,
\end{align*}
with $\Te_n(y',y):=\sum_{j=1}^{m_n}\vp_{n,j}(y')\vp_{n,j}(y)$.
Since the set $\{\eta_k\}_{k\in\mathbb{N}}$ is dense in $H^{\frac{1}{2}}(\pa\Om)$ and
$\chi\equiv1$ on $\Ga_1'$, we deduce $\Te_n(y',y)=\wt\Te_n(y',y)\in L^2(\Ga_1')
\times L^2(\Ga_2)$ for all $n\in\BN$. From \cite[Theorem 1.1]{CanutoKavian:2001}
(see also \cite[Lemma 4.1]{Kian:2022}), we deduce that $m_n=\wt m_n$ and after an
orthogonal transformation
\begin{align}\label{eq:equal varphi}
\vp_{n,j}(y)=\wt\vp_{n,j}(y),\quad j=1,\cdots,m_n,\ \forall y\in\pa\Om,\ n\in\BN.
\end{align}
By \cite[Corollary 1.7]{CanutoKavian:2004}, the equal Dirichlet boundary spectral data
\eqref{eq:equal varphi} imply the desired uniqueness.
\end{proof}

\bibliographystyle{abbrv}

\end{document}